\documentclass[twoside,11pt]{article}

	\usepackage[T1]{fontenc}	
	\usepackage[latin1]{inputenc}	
	\usepackage{vmargin}		
	\usepackage{fancyhdr}		
	\usepackage{amsfonts,amsthm,amsmath,stmaryrd, amssymb,mathrsfs} 
	\usepackage{graphicx, subfigure} 
	\usepackage{mfpic} 
	\usepackage{listings} 
	\setpapersize{custom}{21cm}{29.7cm}
	\setmarginsrb{25mm}{10mm}{25mm}{47mm}{10mm}{6mm}{0mm}{10mm}
	\usepackage[pdftex]{hyperref}
	\usepackage{color}
	\usepackage{bm}
	
	\numberwithin{equation}{section}
	
\def\e{\varepsilon}

\def\R{{\mathbb R}}
\def\C{{\mathbb C}}
\def\N{{\mathbb N}}
\def\Z{{\mathbb Z}}
\def\d{\partial}
\def\a{\alpha}

\def\be{\begin{equation}}
\def\ee{\end{equation}}

\def\op{{\rm op}}
\def\und{\underline}
\def\dx{{\rm d}x}
\def\dxi{{\rm d}\xi}

\begin{document}

\theoremstyle{plain}
\newtheorem{theo}{Theorem}

\theoremstyle{plain}
\newtheorem*{theo*}{Theorem}

\theoremstyle{plain}
\newtheorem{prop}{Proposition}[section]

\theoremstyle{plain} 
\newtheorem{defi}[prop]{Definition}

\theoremstyle{plain}
\newtheorem{coro}[prop]{Corollary}

\theoremstyle{plain}
\newtheorem{lemma}[prop]{Lemma}

\theoremstyle{definition}
\newtheorem*{condi}{Condition}

\theoremstyle{plain}
\newtheorem{notation}[prop]{Notation}

\theoremstyle{plain} 
\newtheorem{remark}[prop]{Remark}

\theoremstyle{plain} 
\newtheorem{hypo}[prop]{Assumption}

\title{On hyperbolicity and Gevrey well-posedness. \\ Part three: a model of weakly hyperbolic systems.}

\author{
Baptiste Morisse
\thanks{School of Mathematics, Cardiff University - \href{mailto:morisseb@cardiff.ac.uk}{morisseb@cardiff.ac.uk}.
The author is supported by the EPSRC grant "Quantitative Estimates in Spectral Theory and Their Complexity" (EP/N020154/1).
The author thanks his advisor Benjamin Texier and Jean-Francois Coulombel for all the remarks on this work.
The author thanks warmly the anonymous referee for all the detailed comments on this work.
}
}

\date{\today}

\maketitle

\begin{abstract}
	We consider a model of weakly hyperbolic systems of first-order, nonlinear PDEs.
	Weak hyperbolicity means here that the principal symbol of the system has a crossing of real-valued eigenvalues, and is not uniformly diagonalizable.
	We prove the well-posedness of the Cauchy problem in the Gevrey regularity for all Gevrey indices $\sigma$ in $(1/2,1)$.
	The proof is based on the construction of a suitable approximate symmetrizer of the principal symbol and an energy estimate in Gevrey spaces.
	We discuss both the generality of the assumption on the structure of the principal symbol and the sharpness of the lower bound of the Gevrey index.
\end{abstract}


\setcounter{tocdepth}{2}
\tableofcontents

\section{Introduction}

In this paper we prove an energy estimate for systems of the form
\be
	\label{4.intro.Cauchy}
	\d_{t} u =
	\begin{pmatrix}
		0 & 1 \\
		(t+ |x-x_0|^2)e(t,x) & 0
	\end{pmatrix}
	\d_xu + F(t,x,u)u
\ee

\noindent where $x\in\R$, $F(t,x,u)$ is nonlinear in $u$, and $e$ is a Gevrey function that is bounded away from zero and compactly supported around $(t,x) = (0,x_0)$.
This result translates by classical arguments into a \textit{local-in-time} well-posedness result in Gevrey spaces for the Cauchy problem for \eqref{4.intro.Cauchy}.

This result could also be extended into a general well-posedness result for a wider class of systems in several spatial dimensions:
\be
	\label{4.intro.local.dt}
	\d_tu = \sum_{1 \leq j \leq d} A_j(t,x) \d_{x_j}u + f(t,x,u)
\ee

\noindent where $x$ in $\R^{d}$, the $A_j$ are in $\R^{2 \times 2}$, $f$ in $\R^{2}$, the $A_j$ have some smoothness in time and are Gevrey regular in $x$, the nonlinearity $f$ is analytic in all variables, and the principal symbol $A = \sum_j A_j(t,x)\xi_j$ experiences a transition from hyperbolicity to ellipticity.
Precisely, in order to extend our result for \eqref{4.intro.Cauchy} into a well-posedness result for \eqref{4.intro.local.dt}, we assume
\begin{itemize}
	\item[$\bullet$] hyperbolicity of the principal symbol $A$, that is the spectrum of $A(t,x,\xi)$ is real.
	\item[$\bullet$] At a distinguished point $(0,x_0,\xi_0) \in\R\times\R^{d}\times\R^d$, the existence of a real and non semi-simple eigenvalue (semi-simplicity means simplicity as a zero of the minimal polynomial of $A(t,x,\xi)$).
	\item[$\bullet$] And finally we assume that $A$ transitions from hyperbolicity to ellipticity at $(0,x_0,\xi_0)$ for negative times.
	By transition from hyperbolicity to ellipticity we mean the phenomenon studied in \cite{morisse2016II}.
	Here this transition is \textit{not} degenerate, we will go back to this point in Section \ref{4.subsection.intro.generic}.
\end{itemize}

\noindent In a forthcoming version of this paper, we expound on these Assumptions, and handle the general case of weakly hyperbolic systems of the form \eqref{4.intro.local.dt}.

In the present version of this paper, we work exclusively with the model \eqref{4.intro.Cauchy}.
The fact that \eqref{4.intro.Cauchy} is one-dimensional ($x\in\R$) does not play any role in our analysis.

\medskip

Further simplifying into $e \equiv 1$, $F(u) = \begin{pmatrix} 0 & 0 \\ 0 & u_1 \end{pmatrix}$, we find the system
$$
	\d_t \begin{pmatrix}
		u_1 \\
		u_2
	\end{pmatrix}
	=
	\begin{pmatrix}
		0 & 1 \\
		t + x^2 & 0
	\end{pmatrix}
	\d_x
	\begin{pmatrix}
		u_1 \\
		u_2
	\end{pmatrix}
	+
	\begin{pmatrix}
		0 \\
		u_1^2
	\end{pmatrix} ,
$$
\noindent which reduces to the wave-like equation in $u_1\in\R$:
\be
	\label{4.intro.wave.nonlin}
	\d_t^2 u_1 = \d_x \left( (t+x^2) \d_x u_1 \right) + \d_{x}(u_1^2)
\ee

\noindent The wave operator in \eqref{4.intro.wave.nonlin} is singular at $(t,x) =(0,0)$, and elliptic for $t + x^2 < 0$ -- in particular for negative times.

\medskip

Our interest is in the Cauchy problem at $t=0$, for forward times.
Our present result has a double background: first in \textit{well-posedness for weakly hyperbolic systems}, a line of research popularized in particular by Colombini and collaborators \cite{colombini1983well}, \cite{colombini2007second} and \cite{bronshtein}, and in \textit{systems transitioning from hyperbolic to ellipticity}, a line of research initiated by Lerner, Morimoto and Xu in \cite{lerner2010instability}.


\subsection{Background: on weakly hyperbolic systems}


\subsubsection{The classical result of Colombini, Janelli and Spagnolo}

We consider here the following second-order, linear scalar equation
\be
	\label{4.intro.wave}
	\d_t^2 v = \d_x \left( a \d_x v \right)
\ee

\noindent with $a=a(t)$ a nonnegative, $C^k([0,T])$ function for some $k \geq 1$.
Such weakly hyperbolic, second-order scalar equations have long been studied by in Gevrey regularity.

A cornerstone of the domain is Colombini, Janelli and Spagnolo's paper \cite{colombini1983well}, which proved Gevrey well-posedness in the case of spatially-independent symbol $a(t)$.
The work of Colombini, Janelli and Spagnolo is based on an energy estimate, which uses the particular structure of the wave equation \eqref{4.intro.wave} and a lemma of real analysis which extends the classical Glaeser's inequality
\footnote{
	In fact, Lemma 1 in \cite{colombini1983well} is a weaker version of Glaeser inequality: Lemma 1 gives a bound on the $L^1$ norm of $a^{1/k}$, whereas the Glaeser inequality is pointwise for $a(t)^{1/2}$.
	},
namely that if $a(t)$ is a $C^k$ nonnegative function on $[0,T]$, then $a(t)^{1/k}$ is absolutely continuous on $[0,T]$ (see Lemma 1 in \cite{colombini1983well}, and \cite{glaeser1963racine} for Glaeser's inequality).

In the case when $a = a(t)$, equation \eqref{4.intro.wave} transforms into the scalar ODE
$$
	\d_t^2 w(t,\xi) = - a(t) |\xi|^2 w(t,\xi)
$$

\noindent thanks to the Fourier transform, and where we denote $w(t,\xi) = \widehat{v}(t,\xi) \in \C$.
As $a(t)$ is supposed to be only nonnegative (weak hyperbolicity), we introduce a small parameter $\e >0$ (later on $\e = \e(\xi)$) and the approximate energy
$$
	E_{\e}(t,\xi) = |\d_t w(t,\xi)|^2 + \left(a(t) + \e \right) |\xi|^2 |w(t,\xi)|^2
$$

\noindent whose time derivative is
$$
	\d_{t} E_{\e} = a'(t) |\xi|^2 |w|^2 + 2 \e |\xi|^2 {\rm Re}\, w \d_t w.
$$

\noindent Having in mind a G\r{a}rding-type inequality to fulfil an energy estimate, we bound the previous equality by
$$
	\d_{t} E_{\e} \leq |a'(t)| |\xi|^2 |w|^2 + \e^{1/2} |\xi| \, E_{\e}
$$

\noindent thanks to Cauchy-Schwarz's inequality.
To bound the term $|a'(t)| |\xi|^2 |w|^2$, we need here to link $|a'|$ to $a+\e$ in order to bound $|a'(t)| |\xi|^2 |w|^2$ by the term $\left(a(t) + \e \right) |\xi|^2 |w(t,\xi)|^2$ of the energy (up to a multiplicative constant).
As $\left( (a+\e)^{1/k} \right)' = \frac{1}{k} a'/(a + \e)^{1-1/k}$, we write
\begin{eqnarray*}
	|a'(t)| |\xi|^2 |w|^2
	& = &
	\left| \frac{a'}{(a+\e)^{1 - 1/k}} \right| \frac{1}{(a+\e)^{1/k}} \, (a + \e) |\xi|^2 |w|^2 \\
	& = &
	k \left| \left( (a+\e)^{1/k} \right)' \right| \frac{1}{(a+\e)^{1/k}} \, (a + \e) |\xi|^2 |w|^2.
\end{eqnarray*}

\noindent As $a$ is nonnegative, there holds
\begin{eqnarray*}
	\d_{t} E_{\e}
	& \leq &
	\left| \left( (a+\e)^{1/k} \right)' \right| \frac{1}{(a+\e)^{1/k}} \, E_{\e} + \e^{1/2} |\xi| \, E_{\e} \\
	& \leq &
	\left| \left( (a+\e)^{1/k} \right)' \right| \e^{-1/k} \, E_{\e} + \e^{1/2} |\xi| \, E_{\e}
\end{eqnarray*}

\noindent hence
\begin{eqnarray*}
	E_{\e}(t,\xi)
	& \lesssim &
	\exp \left( \e^{-1/k} \int_{0}^{t} \left| \left( (a+\e)^{1/k} \right)'(s) \right| ds + t \e^{1/2} |\xi|  \right) E_{\e}(0,\xi) \\
	& \lesssim &
	\exp \left( \e^{-1/k} |a|_{C^k}^{1/k} + T \e^{1/2} |\xi|  \right) E_{\e}(0,\xi)
\end{eqnarray*}

\noindent for all $t \leq T$ thanks to Lemma 1 in \cite{colombini1983well}.
In order to optimize the exponential term, we put $\e = |\xi|^{-2/(k+2)}$ to get finally
$$
	E_{\e}(t,\xi) \lesssim e^{c |\xi|^{2/(k+2)}} E_{\e}(0,\xi)
$$

\noindent for some constant $c>0$.

Thanks to this (pointwise in frequency) energy estimate, the authors of \cite{colombini1983well} proved that the Cauchy problem associated to \eqref{4.intro.wave} is well-posed in Gevrey spaces $\mathcal{G}^{\sigma}_{c}$ (see Definition 2.2 in \cite{morisse2016j}) with $\sigma > 2/(k+2)$, where $k$ is the regularity of the coefficient of equation \eqref{4.intro.wave}.
Note that, as the regularity of $a$ grows, the range of Gevrey indices for which well-posedness holds grows as such.


\subsubsection{Beyond the 1983 article of Colombini, Janelli and Spagnolo}

The work of \cite{colombini1983well} has been followed and extended notably by Colombini and Nishitani in \cite{colombini2007second} and by Colombini, Nishitani and Rauch in \cite{bronshtein}.

In \cite{colombini2007second}, Colombini and Nishitani study the case when $a$ depends also in $x$, that is, $a(t,x)$ is assumed to be nonnegative and in $C^2([0,T], G^{s}_{R})$ where $s$ stands as usual for $1/\sigma$ (see Definition 2.1 in \cite{morisse2016j} for Gevrey spaces defined from the spatial viewpoint, and Proposition 2.1 therein for its link with $\mathcal{G}^{\sigma}_{\tau}$).
Note that, as it is made explicit in Theorem 1.3 in the paper of Colombini and Nishitani, it is assumed that $a(t,x)$ is in fact nonnegative in $[-\delta,T+\delta]$ for some $\delta>0$.
This additional assumption on $a$ is crucial in the course of the proof of \cite{colombini2007second}.
Indeed, in order to extend the energy-based study in \cite{colombini1983well}, the authors of \cite{colombini2007second} use a pseudo-differential calculus.
In the context of symbols, Lemma 1 in \cite{colombini1983well} is no longer helpful, as it leads to an $L^1$ estimate of the time derivative of $a$; instead, a pointwise inequality in $(t,x)$ is needed, hence the use of Glaeser's inequality.
For Glaeser's inequality to hold in a compact subspace of $\R\times\R^{d}$, the nonnegativity condition on $a$ has to hold on a larger subspace containing the compact, see Appendix \ref{4.subsection.glaeser}.
Well-posedness is then proved for any $1 \leq s < 2$ -- that is for any $1/2 < \sigma \leq 1$ thanks to Proposition 2.1 in \cite{morisse2016j} -- extending the work of \cite{colombini1983well}.

More recently, Colombini and Nishitani have pursued their line of research in \cite{colombini2017}.
The authors are interested in wave equations with coefficients with independent variables $t$ and $x$.
Using an exponential weight and the same metric in the phase space as we use in this paper, the authors of \cite{colombini2017} prove again well-posedness for $1/2 < \sigma \leq 1$.

In \cite{garetto2018hyperbolic}, Garetto, J\"ah and Rhuzansky prove well-posedness in anisotropic Sobolev spaces for a large class of linear systems of first-order PDEs.
The authors consider triangular principal symbols and source terms whose order are sufficiently low compared to the dimension of the systems.
Their method is based on representation of solutions of triangular systems.

\medskip

The work of Colombini, Nishitani and Rauch in \cite{bronshtein} explores a different method.
Generic weakly hyperbolic systems \eqref{4.intro.Cauchy} are considered, not only second-order scalar equations \eqref{4.intro.wave} as in \cite{colombini1983well} or \cite{colombini2007second}, i.e. the principal symbol $A(t,x,\xi)$ is there a $N\times N$ matrix with real spectrum but with potential eigenvalue crossings.
To study such general symbols, the authors introduce a \textit{block size barometer} $\theta =m-1$, which roughly measures the extent to which $A(t,x,\xi)$ can be smoothly block diagonalized by blocks of size $m$.
For smoothly diagonalizable symbols, $\theta = 0$ ; on the other hand, $\theta = N-1$ if the symbol is not block diagonalizable at all - which is typically our framework, for $N=2$.
In order to get a general result on well-posedness in Gevrey spaces, regardless of the spectral details of the principal symbol of \eqref{4.intro.Cauchy}, a suitable Lyapunov symmetrizer is studied.
In exchange for a general statement, the range of Gevrey indices for which well-posedness holds is quite reduced, and depends on $\theta$.
Precisely, well-posedness for \eqref{4.intro.Cauchy} is proved for any
$$
	\sigma \geq \min \left\{ \frac{1+6\theta}{2+6\theta} , \frac{2+4\theta}{3+4\theta} \right\} .
$$

\noindent Note that in our framework there holds $\theta = 1$ which leads the lower bound $6/7$ for the Gevrey index.


\subsection{Background: on systems transitioning away from hyperbolicity}


The question of the instability of systems transitioning away from hyperbolicity has been first raised in \cite{lerner2010instability}, extending the work \cite{metivier2005remarks} on initially elliptic systems.
In \cite{lerner2010instability} quasilinear \textit{scalar} equations are considered, with analytic coefficients.
It is assumed that these equations experience a transition from initial hyperbolicity to ellipticity for positive times.
For such equations, it is proved in \cite{lerner2010instability} that the Cauchy problem with initial analytic data is strongly unstable with respect to $C^{\infty}$ perturbation.

A similar instability result is established in \cite{lerner2015onset}, in which quasilinear \textit{systems} with smooth coefficients are considered.
In various cases of transitions from initial hyperbolicity to ellipticity, the Cauchy problem in Sobolev spaces is proved to be unstable, in the sense of Hadamard.
That is, hypothetical flow of the system fails to be H\"older from Sobolev spaces to $L^2$.
The article \cite{lu2016resonances} explores a similar theme in the context of high-frequency solutions of singularly perturbed symmetric hyperbolic systems.

In a previous work \cite{morisse2016II}, we considered first order quasi-linear system \eqref{4.intro.Cauchy} experiencing a transition from hyperbolicity to ellipticity.
A typical example of symbols which falls into the class studied in Section 2.3 in \cite{morisse2016II} is
\be
	\label{4.local.intro.A}
	A(t,x,\xi) =
	\begin{pmatrix}
		0 & 1 \\
		-(t -t_{\star}(x,\xi)) & 0
	\end{pmatrix}
\ee

\noindent in a neighborhood of $(0,0,\xi_0) \in \R_t\times\R^{d}_x\times\R^{d}_{\xi}$, with
\be
	\label{4.intro.t*}
	t_{\star}(x,\xi)= |x|^4 + |\xi - \xi_0|^2
\ee

\noindent In such a case, we proved in Theorem 2.11 in \cite{morisse2016II} that \eqref{4.intro.Cauchy} is not well-posed in Gevrey spaces for $\sigma\in(0,2/13)$.
As explained in Section 2,3 therein, the term $|x|^4$ corresponds to a degenerate time transition.
As we see in Figure \ref{4.pic.diff_x2_x4}, the hyperbolic domain $\left\{ (t,x)\in[0,T] \times B_r(x_0) \, : \, t \leq |x|^4 \right\}$ for $|x|^4$ is thinner than the hyperbolic domain $\left\{ (t,x)\in[0,T] \times B_r(x_0) \, : \, t \leq |x|^2 \right\}$ for $|x|^2$.
This observation allowed us to treat the term $|x|^4$ as a remainder term.
Having treated the case of \textit{degenerate} transitions in our paper \cite{morisse2016II}, we now wish to handle \textit{generic} transitions.
These involve, as explained in \cite{lerner2015onset}, time-transition functions of the form $t_{\star}(x) = x^2$, in one spatial dimension, and a Jordan block for the principal symbol, that is \eqref{4.intro.Cauchy} with $t_{\star}(x) = x^2$.


\begin{figure}[!h]
   \begin{minipage}[c]{.46\linewidth}
      \includegraphics[width=180mm]{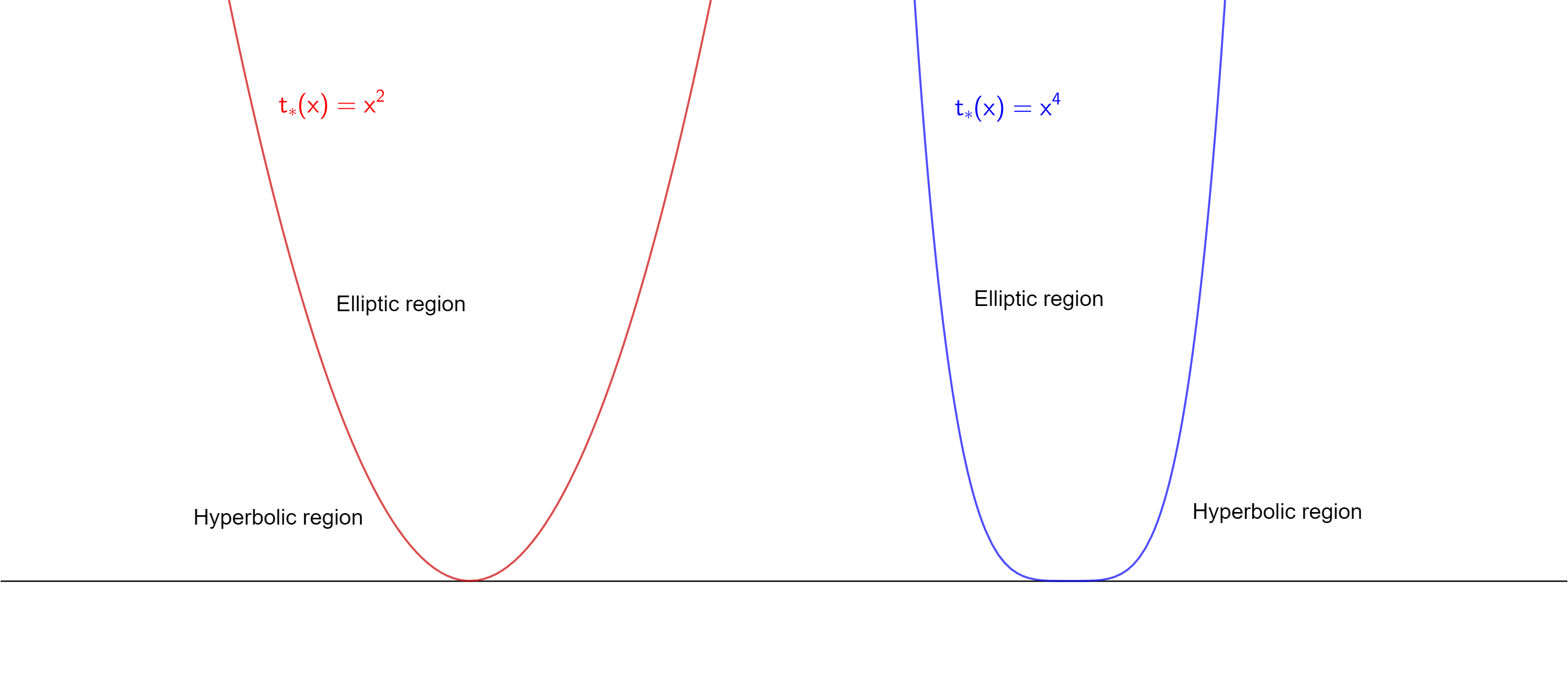}
   \end{minipage} \hfill
	\caption{Comparison between degenerate $x^4$ and non-degenerate $x^2$}\label{4.pic.diff_x2_x4}
\end{figure}


\subsection{Generic time transitions}
\label{4.subsection.intro.generic}


The proof of \cite{morisse2016II} in the case $t_{\star}(x) = |x|^2$ fails essentially due to the size of the hyperbolic domain $\left\{ (t,x)\in[0,T] \times B_r(x_0) \, : \, t \leq |x|^2 \right\}$ in the setting developed therein.
The term $|x|^2$ may not be considered as a remainder term.

Thus in order to prove ill-posedness in the generic configuration, we have to handle the not so small hyperbolic region under the transition curve.
This means proving a form of \textit{well-posedness} for $t < t_{\star}$.
At $t = t_{\star}$ the unstable modes are turned on and the analysis of \cite{morisse2016II} should apply. For the analysis of \cite{morisse2016II} to go through, we must find suitable analytic data $\left(h_{\e} \right)_{\e>0}$ such that the Cauchy problem at $t = t_{\star}$ is ill-posed (with the difficulty that $t_{\star}$ is a
function of $x$ in $1$d and of $(x,\xi)$ in multi-d).

The outstanding question is then to find suitable \textit{initial} (at $t = 0$, for all $x$) data which give rise to the suitable unstable data $ h_{\e}(x)$ at $t = t_{\star}(x)$.
In other words, we want to solve the \textit{backward-in-time} Cauchy problem, in the hyperbolic zone, from $t = t_{\star}(x)$ to $t = 0$.
This motivates the form of the principal symbol under consideration here, as we describe in the next Section.

\subsection{Current result}


As mentioned above, generic transitions from hyperbolic to ellipticity involve in one spatial dimension principal symbols of the form \eqref{4.local.intro.A} with $t_{\star}(x) = x^2$.
In order to study these transitions, we must understand the backward-in-time Cauchy problem for such operators.
This motivates the form of our principal symbol in \eqref{4.intro.Cauchy}.
The function $e$ is assumed to be bounded away from zero and Gevrey (see Assumption \ref{4.hypo.A}).
Under this assumption, we prove an energy estimate for solutions with compact support with regularity $\mathcal{G}^{\sigma}_{\tau}$ for any $\sigma \geq 1/2$ and $\tau>0$ small.
This is Theorem \ref{4.theo}.

The proof relies on the construction of a suitable symmetrizer $S = \op( \text{diag}(1 , b) )$ with symbol $ b(t,x,\xi) = (t + x^2 + \langle \xi \rangle^{-c})^{-1/2} $ and a Gevrey energy estimate.
An important observation is that the symbol $b$ does not belong to a standard class of symbols.
Indeed, $b(0,0,\xi) = \langle \xi \rangle^{c/2}$ whereas $b(t,x,\xi) \in S^{0}_{1,0} $ when $t \geq \und{t}$ and $|x| >r$.
To reconcile both point of views, we make use of class of symbols defined with respect to a metric of the phase space, as described in \cite{lerner2011metrics}.
In Lemma \ref{4.lemma.b.symbol}, we prove that $b \in S(b,g)$ where the time-dependent metric $g$ is defined in \eqref{4.def.metric}.
This metric has already been used in {\rm \cite{colombini2007second}} and \cite{colombini2017}.
Our paper relies also on our paper \cite{morisse2016j} which contains our work on pseudo-differential operators with symbols which are Gevrey regular in the spatial variable.

%
%

We note that the symmetrizer $S$ is anisotropic, as it stresses out more (Sobolev) regularity for the second component $u_2$ than for the first one $u_1$.
This observation is closely related to \cite{garetto2018hyperbolic}, in which well-posedness is proved in anisotropic Sobolev spaces for a certain class of weakly hyperbolic systems.
In \cite{garetto2018hyperbolic}, the additional assumption on the order of the force terms compared to the dimension of the systems can be read as $c=2$ in our settings, where $c$ is the order of the perturbation.
In the present paper, without such a strong assumption on the source term $F$ we cannot expect to reach $c=2$ but rather $c \in (0,1]$.
See also Remark \ref{remark.hypo.nonlinear}, which gives a hint on how to reach $c=4/3$.

\begin{remark}
	Our result is outside the range of the article {\rm \cite{colombini2007second}}.
	The symbol $a(t,x)$, which is in our case similar to $t + x^2$, does not satisfy Glaeser's inequality for negative times. This result is also an improvement of the result given in {\rm \cite{bronshtein}}, as we attain in our paper the lower bound $1/2$ for the Gevrey indices, compared to the lower bound $6/7$ as described above.
	The main difference is that, in our paper, we take care of the spectral details of the principal symbol, as we assume it is a $2$ by $2$ matrix, with a specific crossing of eigenvalues.
\end{remark}

\section{Main assumptions and results}


We consider the Cauchy problem \eqref{4.intro.Cauchy} which we rewrite in a more compact way as
\be
	\label{4.system}
	\d_t u =
	\begin{pmatrix}
		0 & 1 \\
		a(t,x) & 0 
	\end{pmatrix}
	\d_x u + F(t,x,u)u
\ee

\noindent where $x$ is in $\R$, $u$ in $\R^{2}$ and $F(t,x,u)$ is a $2\times 2$ matrix. 
We describe first our assumptions on the regularity and the structure of both $a$ and $F$.



\begin{hypo}[Structure and regularity for $a$]
	\label{4.hypo.A}
	We assume that 
	$$
		a(t,x) = \left(t+|x-x_0|^2\right)e(t,x)
	$$
	
	\noindent where $e(t,x)$ has compact support $[0,T']\times B_{r'}(x_0)$ for some $T'>0$ and $r'>0$. 
	Besides, $e$ is in $C^2([0,T'],G^{1/\sigma}_R)$, that is there is $C>0$ such that
	$$
		\left| \d_{x}^{\a}e(t,x) \right| \leq C R^{|\a|} \a!^{1/\sigma} \quad , \quad \forall \, \a \in \N \; , \; \forall \, (t,x) \in [0,T']\times B_{r'}(x_0) .
	$$
	
	\noindent There is also $0 < T < T'$ and $ 0 < r < r'$ such that
	\be
		\label{4.ineq.e}
		1/2 \leq e(t,x) \leq 2 \quad , \quad \forall \,(t,x) \in [0,T]\times B_{r}(x_0) .
	\ee
\end{hypo}

\noindent We denote
\be
	\label{4.def.und.tau}
	\und{\tau} = R^{-\sigma}/\sigma
\ee

\noindent the Gevrey regularity of $a$, in the Fourier point of view (see Definition 2.2 and Proposition 2.1 in \cite{morisse2016j} with $s = 1/\sigma$).
Concerning the force term $F(t,x,u)$, we make the following

\begin{hypo}[Regularity for $F$]
	\label{hypo.reg.F}
	The function $F$ is entire in $u$ in a neighborhood of $u=0$ and there holds
	\be
		F(t,x,u) = \sum_{k \in \N^2} F_k(t,x) u^k 
	\ee	
	
	\noindent where coefficients $F_k(t,x)$ are in $\mathcal{G}^{\sigma}_{\tau_0}$, uniformly in $t$ and $k \in \N^2$. 
\end{hypo}

\noindent As spaces $\mathcal{G}^{\sigma}_{\tau}$ are algebra, if $u$ is controlled in $\mathcal{G}^{\sigma}_{\tau}$ the same holds for all powers $u^k$.
We could lighten the assumption of analyticity in the variable $u$ for $F$ by assuming some Gevrey regularity.
This would only add technicalities, which we choose to avoid at this stage.

\medskip

The main result of our paper is an energy estimate in Gevrey space ${\rm \mathcal{G}}^{\sigma}_{\tau}$ for any $\sigma \geq 1/2$ and for small $\tau$. 
The lower Gevrey index $1/2$ is the expected lower bound for the Gevrey regularity in the presence of a source term $F(t,x,u)u$.
With additional assumption on $F$, the same analysis may lead to a lower bound $\sigma \geq 1/3$ (see Remark \ref{remark.hypo.nonlinear}). 
To obtain such a result, we define a suitable symmetriser for $A$, introducing first the symbol
\be
	\label{4.def.b}
	b(t,x,\xi) = \left( a(t,x) + \langle \xi \rangle^{-c} \right)^{-1/2}
\ee

\noindent for some $c\in(0,2]$ and denoting 
\be
	\langle \xi \rangle = \left( 1 + |\xi|^2 \right)^{1/2} 
\ee
 
Defining 
\be 
	\label{def.symbol.S}
	S(t,x,\xi)
	=
	\begin{pmatrix}
		1 & 0 \\
		0 & b(t,x,\xi)
	\end{pmatrix}
\ee 

\noindent one key point is that 
\be
	S^2
	\begin{pmatrix}
		0 & 1 \\
		a + \langle \cdot \rangle^{-c}
	\end{pmatrix}
	=
	\begin{pmatrix}
		0 & 1 \\
		1 & 0
	\end{pmatrix}
\ee 

\noindent is real symmetric.
The perturbation by a lower order term implies working in Gevrey regularity, but in exchange allows for an approximate symmetrization of the principal symbol $A$.
This is closely related to the work of Colombini and M\'etivier \cite{colombini_metivier} for uniformly diagonalizable symbols, depending only on time.

Section \ref{4.subsection.suitable.class} will be devoted to prove that $b$ is in the class of symbols $S(b,g)$, defined in \eqref{4.defi.class.symbols} and the metric $g$ defined in \eqref{4.def.metric}. 
This is done principally thanks to the non-negativity of $a$ and Glaeser's inequality (see Lemma \ref{4.lemma.glaeser.und.a} and Section \ref{4.subsection.glaeser} below). 

In all the following, we denote
\be
	\label{4.def.D}
	D = {\rm op}\left( \langle \xi \rangle \right) \quad \text{and} \quad D^{\sigma} = {\rm op}\left( \langle \xi \rangle^{\sigma} \right) .
\ee

\noindent Let $\sigma \in (0,1)$, $\tau > 0$ and $u$ in ${\rm \mathcal{G}}^{\sigma}_{\tau}$. 
We introduce the Gevrey energy
\be
	\label{4.def.gevrey.energy.sys}
	E(\tau,u(t)) = \frac{1}{2} \left| \op(S)e^{\tau D^{\sigma}}u(t) \right|^2_{L^2} .
\ee

Thanks to the result of sharp finite speed propagation for \eqref{4.system} under assumptions of "constant outside a compact set", the result of \cite{colombini2010sharp} can be used. 
We look for solutions with compact support in $(t,x)$ included in $[0,T] \times B_{r}(x_0)$, which can be done if the initial datum $u_0$ has sufficiently small compact support (with respect to $T$ and the finite speed propagation of \eqref{4.system}). 
The existence of such solutions with regularity in $\mathcal{G}^{\sigma}_{\tau}$ is assured by our main result and by standard results on local well-posedness in Gevrey for such systems.
\begin{theo}
	\label{4.theo}
	For any $\tau_0 < \und{\tau}$ with $\und{\tau}$ defined in \eqref{4.def.und.tau}, there is ${\bm \tau} >0$ such that 
	$$
		E\left(\tau_0 - {\bm \tau} t , u(t) \right) \lesssim E(\tau_0, u(0) ) \quad , \quad \forall \,t \in \left[ 0,\min \left( T, \frac{\tau_0}{{\bm \tau}} \right) \right] .
	$$
\end{theo}

Section \ref{4.section.energy} is devoted to the proof of Theorem \ref{4.theo}.
Here are some remarks concerning our result:

\begin{itemize}
	\item Concerning the case $x \in \R^d$ for $d \geq 2$, that is for 
	$$
		\d_t u = \sum_{1 \leq j \leq d} A_j(t,x) \d_{x_j}u + F(t,x,u)u
	$$
	\noindent our method described in the present paper may also apply. 
	Considering the principal symbol $A(t,x,\xi) = \sum_{1 \leq j \leq d} A_j(t,x) i\xi_j$, the analogous of \eqref{4.system} is for the principal symbol to have normal form
	$$
		A(t,x,\xi) = 
		i|\xi|
		\begin{pmatrix}
		0 & 1 \\
		a(t,x,\omega) & 0 
	\end{pmatrix}
	\quad \text{with} \quad 
	\omega = \frac{\xi}{|\xi|} 
	$$
	
	\noindent with $a(t,x,\omega) \sim t + x^2$.
	
	\item For higher dimensions for the system, our method may also apply for normal forms
	$$
		\begin{pmatrix}
			0 & 1 & 0 & \hdotsfor{3} \\
			0 & 0 & 1 & 0 & \hdotsfor{2} \\
			\hdotsfor{6} \\
			a_1 & a_2 & \hdotsfor{2} & a_{N-1} & 0
		\end{pmatrix}
	$$
	\noindent We would need to adapt consequently our symmetrizer $S$ as 
	$$
		\begin{pmatrix}
			0 & 1 & 0 & \hdotsfor{3} \\
			0 & 0 & 1 & 0 & \hdotsfor{2} \\
			\hdotsfor{6} \\
			a_1 + \langle \xi \rangle^{-c} & a_2 & \hdotsfor{2} & a_{N-1} & 0
		\end{pmatrix}
	$$
	
	\noindent and the expected lower bound for the Gevrey index is then $1 - 1/N$ (again, in the presence of source term $Fu$).
	
	\item The general case $a(t,x) \geq 0$, with $C^2$ regularity, could be treated by our method, even in the case where $\d_t a(0,x_0) =0$.
	The symmetrizer $S$ could be defined in the same way.
	The symbol $b$ is still in $S(b,g)$, thanks to Glaeser inequality.
	The only main difference would be the care of the term $\d_tS$ in the energy estimate.
	In the case $\d_t a(0,x_0) =0$, a time Glaeser inequality holds which allows to control $\d_ta$ in terms of $b^{-1}$. 
\end{itemize}

Finally we note that the main case of fully quasilinear systems $A = A(u)$, such as Euler equations with Van der Waals laws or other physical meaningful systems, are out of reach of our current analysis and understanding.
The study of such systems should be a main topic in the future of our research.
\section{Proof of the energy estimate}
\label{4.section.energy}


In order to study \eqref{4.system} in Gevrey spaces, a classical approach is to introduce a Gevrey radius $\tau(t)$ which decreases linearly in time. Let $\tau_0 < \und{\tau}$.
We define
\be
	\label{4.def.tau.t}
	\tau(t) = \tau_0 - {\bm \tau} t.
\ee

\noindent with ${\bm \tau}>0$ to be determined in the course of the proof.


\subsection{Time derivative of the energy}
\label{4.subsection.derivative.energy}



We compute here the time derivative of the energy $E$ defined in \eqref{4.def.gevrey.energy.sys}.
The energy $E$ depends on time through the symbol $b$, the Gevrey radius $\tau(t)$ and $u$.

We introduce
\be
	\label{4.def.v}
	v(t) = e^{ \tau(t) D^{\sigma} } u(t)
\ee

\noindent with $\tau(t)$ defined in \eqref{4.def.tau.t} and $D^{\sigma}$ in \eqref{4.def.D}.
There holds
$$
	\d_{t} v(t) = - {\bm \tau} D^{\sigma} v(t) + e^{ \tau(t) D^{\sigma} } \d_{t}u(t).
$$

\noindent As $u$ solves system \eqref{4.system}, $v$ solves
$$
	\d_t v = - {\bm \tau} D^{\sigma} v + e^{ \tau D^{\sigma} } \left(A \d_{x} u + F(u)u \right).
$$

\noindent In order to work with $v$, we use Notation (3.1) in \cite{morisse2016j} for the conjugation operator of a Gevrey function, writing
\be
	\label{dtv}
	 \d_t v = - {\bm \tau} D^{\sigma} v +\left(A^{(\tau)} \d_{x} v + F(u)^{(\tau)}v \right)
\ee

\noindent where the coefficients of matrices $A^{(\tau)}$ and $F^{(\tau)}$ are the Gevrey conjugated coefficients of $A$ and $F(u)$.

We compute the time derivative of the energy $E(\tau(t),u(t))$ defined in \eqref{4.def.gevrey.energy.sys}.
Using notation $v$ defined in \eqref{4.def.v}, the energy is
$$
	E(t,u(t)) = \frac{1}{2} \left| \op(S) v \right|^2_{L^2}
$$

\noindent where the symbol $S$ is defined in \eqref{def.symbol.S}.
Denoting here $ \langle \cdot \rangle $ the $L^2(\R^{d})$ scalar product, we compute
\begin{eqnarray*}
	\d_t E & = & {\rm Re}\, \langle \op(S) \d_t v , \op(S) v \rangle + {\rm Re}\, \langle \op(\d_t S) v , \op(S) v \rangle
\end{eqnarray*}

\noindent Using \eqref{dtv} there holds
\be
	\d_{t} E = {}-{\bm \tau} {\rm E_{1}} + {\rm E_{2}} + {\rm E_{3}} + {\rm E_{4}}
\ee

\noindent where
\be
	\label{4.FI}
	{\rm E_{1}} = {\rm Re}\, \langle \op(S) D^{\sigma} v, \op(S) v \rangle
\ee

\noindent is the time-derivative of the Gevrey weight ;
\be
	\label{4.FII}
	{\rm E_{2}} = {\rm Re}\, \langle \op(S) A^{(\tau)} \d_{x} v, \op(S) v \rangle
\ee

\noindent are linear terms in the equations ;
\be
	\label{4.FIII}
	{\rm E_{3}} = {\rm Re}\, \langle \op(\d_{t}S) v, \op(S) v \rangle
\ee

\noindent is the time-derivative of the symmetrizer ;
\be
	\label{4.FIV}
	{\rm E_{4}} = {\rm Re}\, \langle \op(S) F(u)^{(\tau)} v , \op(S) v \rangle
\ee

\noindent are the non-linear terms in the equation.
The term ${\rm E_{1}}$ is of higher order than the energy, thanks to the $D^{\sigma}$ term coming from the time derivative of the Gevrey weight.
The minus sign in front of ${\rm E_{1}}$ is crucial in order to control the remainder terms ${\rm E_{2}}$, ${\rm E_{3}}$ and ${\rm E_{4}}$.
We focus now on each of those terms.

\subsubsection{The term ${\rm E_1}$}
\label{subsubsection.E1}

The term ${\rm E_1}$ controls more than the $L^2$-norm of $\op(S)v$.
Component-wise, we get
$$
	{\rm E_{1}} = {\rm Re}\, \langle D^{\sigma} v_1, v_1 \rangle + {\rm Re}\, \langle \op(b) D^{\sigma} v_2, \op(b) v_2 \rangle .
$$

\noindent The anisotropy of the symetriser $S$ reads in this equality.
The first term is simply equals to $ \left|D^{\sigma/2} v_1 \right|^2_{L^2}$: we have a control of the $H^{\sigma/2}$-norm of $v_1$.

Concerning the component $v_2$, we compute
\begin{eqnarray*}
	{\rm Re}\, \langle \op(b) D^{\sigma} v_2, \op(b) v_2 \rangle
	& = &
	{\rm Re}\, \langle D^{\sigma} \op(b) v_2, \op(b) v_2 \rangle + {\rm Re}\, \langle \left[ \op(b) , D^{\sigma} \right] v_2, \op(b) v_2 \rangle \\
	& = &
	\left|D^{\sigma/2} \op(b) v_2 \right|^2_{L^2} + {\rm Re}\, \langle \left[ \op(b) , D^{\sigma} \right] v_2, \op(b) v_2 \rangle .
\end{eqnarray*}

\noindent Focusing on the commutator term on the right-hand side, we write
\begin{eqnarray*}
	{\rm Re}\, \langle \left[ \op(b) , D^{\sigma} \right] v_2, \op(b) v_2 \rangle
	& = &
	{\rm Re}\, \langle D^{-\sigma/2}\left[ \op(b) , D^{\sigma} \right] \op(b)^{-1} D^{-\sigma/2} \, D^{\sigma/2} \op(b) v_2, D^{\sigma/2} \op(b) v_2 \rangle \\
	& = &
	{\rm Re}\, \langle \op \left( S \left( \lambda^{-2} , g \right) \right) \, D^{\sigma/2} \op(b) v_2, D^{\sigma/2} \op(b) v_2 \rangle
\end{eqnarray*}

\noindent thanks to Lemma \ref{4.lemma.composition}.
As $c \in (0,2]$, there holds $\lambda^{-2} \leq 1$ hence operators $\op \left( S \left( \lambda^{-2} , g \right) \right)$ are bounded in $L^2$ using Lemma \ref{4.lemma.action}.
We conclude thus by
\be
	\label{E1.v2}
	 {\rm E_1} \sim \left|D^{\sigma/2} v_1 \right|^2_{L^2} + \left|D^{\sigma/2} \op(b) v_2 \right|^2_{L^2} .
\ee

\subsubsection{The term ${\rm E_{2}}$}

The crucial cancellations take place here.
They rely on our choice of $b$ defined in \eqref{4.def.b}.
As $a$ is in ${\rm \mathcal{G}}^{\sigma}_{\und{\tau}}$ with $\und{\tau}$ defined in \eqref{4.def.und.tau} and by the results of Section 5 in \cite{morisse2016j} (see also \cite{bronshtein}), there is a symbol $\tilde{a}$ in $S^{0}_{1,0}$ such that
\be
	\label{4.def.tilde.a}
	a^{(\tau)} = {\rm op}(\tilde{a})
\ee

\noindent for all $\tau = \tau(t)$, as $\tau(t) \leq \tau_0 < \und{\tau}$ by definition \eqref{4.def.tau.t}.
Denoting
$$
	\tilde{A}
	=
	\begin{pmatrix}
		0 & 1 \\
		\tilde{a} & 0
	\end{pmatrix}
$$

\noindent we may then write
\begin{eqnarray*}
	A^{(\tau)}
	& = &
	\op(\widetilde{A}) \\
	& = & A + \op\left( \widetilde{A} - A \right) \\
	& = & \op\left( A_{\natural} \right)
	-
	\begin{pmatrix}
		0 & 0 \\
		D^{-c} & 0
	\end{pmatrix}
	+ \op\left( \widetilde{A} - A \right)
\end{eqnarray*}
\noindent where $D$ is defined in \eqref{4.def.D}.

We make use of this decomposition to write
\be
	\label{E2.first}
	{\rm E_{2}} = {\rm Re}\, \langle \op(S) \op\left( A_{\natural} \right) \d_{x} v, \op(S)v \rangle + R_2
\ee

\noindent where $R_2$ comprises remainder terms:
$$
	R_2 = -{\rm Re}\, \langle \op(S) \left(
	\begin{pmatrix}
		0 & 0 \\
		D^{-c} & 0
	\end{pmatrix}
	 + \op\left( \widetilde{A} - A \right) \right) \d_{x} v, \op(S)v \rangle.
$$

\noindent As $S$ is defined as a microlocal symmetriser for $A_{\natural}$, we write
$$
	{\rm E_{2}} = {\rm Re}\, \langle \op(S)^2 \op(A_{\natural}) \d_{x} v, v \rangle + R_2 ,
$$

\noindent as $\op(S)^* = \op(S)$ in Weyl quantization for diagonal matrices with real symbols.

Next, applying equality \eqref{composition.crochet} of Lemma \ref{4.lemma.composition} on composition of operators, there holds
$$
	{\rm op}(b)^2 = \op\left( b^2\right) + \op\left( S\left(b^2\lambda^{-2} , g \right) \right)
$$

\noindent as $\{ b , b \} = 0$.
Thus
$$
	{\rm E_{2}} = {\rm Re}\, \langle \op \left( S^2 A_{\natural} \right) \d_{x} v, v \rangle + \widetilde{R}_2
$$

\noindent where
$$
	\widetilde{R}_2 = R_2 + {\rm Re}\, \langle \left( \op\left( S^2 A_{\natural} \right) - \op(S)^2 \op(A_{\natural}) \right) \d_{x} v, v \rangle .
$$

By definition of $S$, the leading term of ${\rm E_2}$ cancels and there holds finally
\begin{eqnarray}
	{\rm E_{2}}
	& = &
	- {\rm Re}\, \langle \op(S)
	\begin{pmatrix}
		0 & 0 \\
		D^{-c} & 0
	\end{pmatrix}
	\d_{x} v, \op(S) v \rangle \label{4.FII.1} \\
	& &
	\quad + {\rm Re}\, \langle \left( \op\left( S^2 A_{\natural} \right) - \op(S)^2 \op(A_{\natural}) \right) \d_x v, v \rangle \label{4.FII.2} \\
	& &
	\quad + {\rm Re}\, \langle \op(S) \op(\widetilde{A} - A) \d_{x} v, \op(S)v \rangle. \label{4.FII.3}
\end{eqnarray}

\subsubsection{The term ${\rm E_{3}}$}

We first note that
$$
	\d_t S
	=
	\begin{pmatrix}
		0 & 0 \\
		0 & \d_t b
	\end{pmatrix}
$$

\noindent and that $ \d_{t} b = -\frac{1}{2} \d_{t} a \, b^3 $.
Thanks to Assumption \ref{4.hypo.A}, function $\d_{t}a(t,x)$ is positive.
We may then write
\be
	\label{4.development.dtb}
	\d_{t} b = -\frac{1}{2} \left( \sqrt{\d_{t} a} \, b \right)^{2} b
\ee

\noindent and aim at getting a G\r{a}rding-type estimate.
As $\sqrt{\d_t a}$ depends only on $(t,x)$ variables, it is in $ S(1,g) $, hence $ \sqrt{\d_{t} a} \, b  $ is in $S\left(b,g\right)$ by Lemma \ref{4.lemma.algebra}.
First, applying equality \eqref{composition.crochet} of Lemma \ref{4.lemma.composition} and as $\{ \sqrt{\d_{t} a} \, b , \sqrt{\d_{t} a} \, b \} = 0$, there holds
$$
	\op \left( \left( \sqrt{\d_{t} a} \, b \right)^{2} \right) = \op \left( \sqrt{\d_{t} a} \, b \right)^{2} + \op \left( S\left( b^2 \lambda^{-2} ,g \right) \right) .
$$

\noindent Second, using again Lemma \ref{4.lemma.composition} there holds
$$
	\op\left(\d_t b\right) = -\frac{1}{2} \op\left( \left( \sqrt{\d_{t} a} \, b \right)^{2} \right) \op(b) + \op \left( i \{ \d_t a \,b^2 , b \} \right) + \op\left( S\left( b^3\lambda^{-2} , g \right) \right).
$$

\noindent The subprincipal symbol $i \{ \d_t a \,b^2 , b \}$ is \textit{a priori} in $S(b^3 \lambda^{-1},g)$.
A careful computation gives however
\begin{eqnarray*}
	i \{ \d_t a \,b^2 , b \}
	& = &
	i(\d_{tx}a) b^2 \, \d_{\xi}b \\
	& = &
	-\frac{i}{2} (\d_{tx}a) b^5 \d_{\xi} \langle \cdot \rangle^{-c}
\end{eqnarray*}

\noindent which is in $S(b^5 \langle \cdot \rangle^{-c-1} , g)$.
We conclude then that
\be
	\label{subprincipal.dtb}
	\op\left(\d_t b\right)
	=
	-\frac{1}{2} \op \left( \left( \sqrt{\d_{t} a} \, b \right) \right)^{2} \op(b) + \op \left( S(b^5 \langle \cdot \rangle^{-c-1} , g) \right) + \op\left( S\left( b^3\lambda^{-2} , g \right) \right).
\ee

\noindent Note that both terms are not comparable, as $c$ may go from $0$ to $2$.

This implies
\begin{eqnarray*}
	{\rm E_{3}}
	& = &
	- \frac{1}{2} {\rm Re}\, \langle \left( \op\left( \sqrt{\d_{t} a} \, b \right) \right)^{2} \op(b) v_2, \op(b) v_2 \rangle \\
	& &
	\quad + {\rm Re}\, \langle \op \left( S(b^5 \langle \cdot \rangle^{-c-1} , g) \right)  v_2, \op(b) v_2 \rangle + {\rm Re}\, \langle \op\left( S\left( b^3\lambda^{-2} , g \right) \right)  v_2, \op(b) v_2 \rangle .
\end{eqnarray*}

\noindent The first term in the above right-hand side satisfies
\begin{eqnarray*}
	{\rm Re}\, \langle \left( \op\left( \sqrt{\d_{t} a} \, b \right) \right)^{2} \op(b) v_2, \op(b) v_2 \rangle
	& = &
	\left| \op\left( \sqrt{\d_{t} a} \, b \right) \op(b) v_2 \right|^2 \\
	& \geq &
	0 .
\end{eqnarray*}

\noindent Thus
\be
	{\rm E_{3}} \leq {\rm Re}\, \langle \op \left( S(b^5 \langle \cdot \rangle^{-c-1} , g) \right)  v_2, \op(b) v_2 \rangle + {\rm Re}\, \langle \op\left( S\left( b^3\lambda^{-2} , g \right) \right)  v_2, \op(b) v_2 \rangle . \label{4.FIII.1}
\ee

\noindent Note that the term ${\rm E_{3}}$ does not depend on $v_1$ as the symbol $S$ is anisotropic.

\section{Pseudo-differential tools}


This Section aims to remind a few tools of pseudo-differential calculus and of symbols associated to a general metric of the phase space.
We start first by study the symbol $b$, which leads naturally to a specific metric $g$ that encodes the specific dynamics of our system.
We then define classes of symbols $S(M,g)$, and the properties of pseudo-differential operators associated to such symbols.


\subsection{Study of symbol $b$}
\label{4.subsection.suitable.class}


We define the symbol
\be
	\label{def.a.natural}
	a_{\natural} = a_{\natural}(t,x,\xi) = a(t,x) + \langle \xi \rangle^{-c}
\ee

\noindent where the additional term $\langle \xi \rangle^{-c} $ makes the symbol $a_{\natural}$ positive.
This is a standard approach when dealing with weakly hyperbolic equations, see \cite{colombini1983well}.
Thanks to this notation, we may write the symbol $b$ defined by \eqref{4.def.b} as $b = a_{\natural}^{-1/2} $.

\begin{lemma}[Bounds for $b$]
	The symbol $b$ satisfies the upper bound
	\be
		\label{4.upper.bound.b}
		b(t,x,\xi) \leq \langle \xi \rangle^{c/2} \quad , \quad \forall \, (t,x,\xi) \in[0,T]\times B_{r}(x_0) \times \R
	\ee
	\noindent and is bounded from below
	\be
		\label{4.lower.bound.b}
		b(t,x,\xi) \gtrsim 1 \quad , \quad \forall \, (t,x,\xi) \in[0,T]\times B_{r}(x_0) \times \R
	\ee
\end{lemma}

\begin{proof}
	The proof of the upper bound \eqref{4.upper.bound.b} is immediate as $a$ is non negative.
	For the lower bound, there holds
	$$
		b(t,x,\xi) \geq \left( \sup a + 1 \right)^{-1/2}
	$$

	\noindent where the $\sup$ of $a$ is over $[0,T]\times B_{r}(x_0)$.

\end{proof}

In order to compute carefully some estimates on the derivatives of $b$, we prove first a local Glaeser inequality for $a$, as it is non-negative locally around $x=x_0$.

\begin{lemma}[Glaeser inequality for $a$]
	\label{4.lemma.glaeser.und.a}
	Under Assumption {\rm \ref{4.hypo.A}}, there is a neighborhood $[0,T]\times B_{r}(x_0)$ of $(0,x_0)\in\R_{t}\times\R_{x}$ and a constant $C_{T,r}>0$ for which there holds
	\be
		\left( \d_{x}a(t,x) \right)^2 \leq C_{T,r} \,a(t,x) \quad , \quad \forall \,(t,x)\in [0,T]\times B_{r}(x_0) .
	\ee
\end{lemma}

\noindent The proof is postponed to Appendix \ref{4.subsection.glaeser}.
The following Lemma gives precise estimates on the derivatives of $b$.

\begin{lemma}[Derivatives of the symbol $b$]
	\label{4.lemma.derivatives.b}
	There is a bounded sequence of constants $C_{\a,\beta}>0$ for which there holds
	\be
		\label{4.derivatives.b}
		|\d_{x}^{\a} \d_{\xi}^{\beta} b(t,x,\xi)| \leq C_{\a,\beta} \widetilde{R}^{|\a|} 2^{|\beta|} \, \a!^{1/\sigma} \beta! \, b(t,x,\xi)\,b(t,x,\xi)^{|\a|} \,\langle\xi\rangle^{-|\beta|} \quad , \quad \forall \,(\a,\beta) \in \N \times \N
	\ee

	\noindent for all $(t,x)$ in $[0,T]\times B_{r}(x_0)$ and $\xi$ in $\R$, and where $\widetilde{R} = R(1 + |a|_{1/\sigma,R}|)> R$.

\end{lemma}

The proof is postponed in Appendix \ref{4.subsection.glaeser}.
It relies on the Fa\`a di Bruno formula (see Lemma \ref{4.lemma.faa}) and the Glaeser inequality for $a$ proved in Lemma \ref{4.lemma.glaeser.und.a}.
We follow through with some remarks on this result.

\begin{remark}
	Thanks to inequality \eqref{4.upper.bound.b}, Lemma {\rm \ref{4.lemma.derivatives.b}} implies that $b \in S^{c/2}_{1,c/2}G^{1/\sigma}_{\widetilde{R}}
	$, as defined in {\rm \cite{morisse2016j}}.
	Without the Glaeser inequality described in Lemma {\rm \ref{4.lemma.glaeser.und.a}}, we would only prove that $ b \in S^{c/2}_{1,c}G^{1/\sigma}_{\widetilde{R}} $, whereas $c$ may be in $[1,2]$.

	The importance of the Glaeser inequality explains why we do not define $b$ as $\left(\widetilde{a}+ \langle \xi \rangle^{-c}\right)^{-1/2}$ where $\widetilde{a}$ is defined in \eqref{4.def.tilde.a} as the symbol of operator $a^{(\tau)}$, the Gevrey conjugation of $a$.
	Indeed the symbol $\widetilde{a}$ does not satisfy a priori the Glaeser inequality, as it is not real.
\end{remark}

\begin{remark}
	As $a$ has compact support, $b(\cdot, \xi)$ is constant outside a compact set of $\R_t\times \R_{x}$ which does not depend on $\xi$.
\end{remark}

The bounds \eqref{4.derivatives.b} show in particular that the symbol $b$ has a variable order and a varying "class" with respect to time and space.
Indeed, for $(t,x) = (0,x_0)$, symbol $a_{\natural}$ is equal to $\langle \xi \rangle^{-c}$, hence $b(t=0)$ is likely to be in the class of classical symbols $S^{c/2}_{1,c/2}$.
But as time goes, the order of $b$ decreases.
In fact, for $t\geq \und{t} >0$, there holds simply $a_{\natural} \geq t \geq \und{t}$, hence
$$
	|\d_{x}^{\a} \d_{\xi}^{\beta} b(t,x,\xi)| \leq C_{\a,\beta} \widetilde{R}^{|\a|} \, \a!^{1/\sigma} \beta! \, \und{t}^{-1/2} \,\und{t}^{-|\a|/2} \,\langle\xi\rangle^{-|\beta|}
$$

\noindent for all $t \geq \und{t}$.
Then $b$ is in the classical space of symbols $S^{0}_{1,0}$ for all $t \geq \und{t}$.

\medskip

A way to reconcile both points of view is to introduce the following time-dependent, non-flat metric in the phase space
\be
	\label{4.def.metric}
	g_{(x,\xi)}(\dx,\dxi)  = \frac{|\dx|^2}{a_{\natural}(t,x,\xi)} + \frac{|\dxi|^2}{\langle\xi\rangle^2} .
\ee

\noindent Lemma \ref{4.lemma.derivatives.b} reads now as
\be
	\label{b.Sbg}
	b \in S(b,g)
\ee

\noindent where $S(b,g)$ is defined in Definition \ref{4.defi.class.symbols}.
Both the weight and the metric are time-dependent, hence encoding precisely the dynamic of the system.

\subsection{Properties of class of symbols}

Properties of pseudo-differential calculus come directly from properties of the metric and the weights associated to the metric.
We give here the fundamental statements about the metric and the weights, which are necessary for a pseudo-differential calculus to be coherent (see Lemma \ref{4.lemma.composition}).
For the sake of simplicity and completeness, we choose to postpone definitions and proofs in the Appendix \ref{4.subsection.metrics}.

\begin{lemma}
	\label{4.lemma.g.admissible}
	The metric $g$ defined in \eqref{4.def.metric} is an admissible metric.
\end{lemma}

\noindent See Lemma \ref{lemma.admissibility.proof} in the Appendix and its proof for further details.

\begin{lemma}
	\label{4.lemma.admissible.weights}
	For all $k \in \Z$, symbols $b^k$ are admissible with respect to the metric $g$.
	For all $m \in \R$, symbols $\langle \xi \rangle^m$ are admissible with respect to the metric $g$.
\end{lemma}

\noindent In particular, Lemma \ref{4.lemma.derivatives.b} implies
\begin{lemma}
	\label{4.lemma.b.symbol}
	For any $k$ in $\Z$, the symbol $b^{k}$ is in $S(b^{k}, g)$.
\end{lemma}

\begin{proof}
	The case $k=1$ is just Lemma \ref{4.lemma.derivatives.b}.
	Hence the result for any $k \geq 1$, thanks to Lemma \ref{4.lemma.algebra}.
	The case $k=-1$ is proved by the same proof as Lemma \ref{4.lemma.derivatives.b}, as $b^{-1} = a_{\natural}^{1/2}$.
\end{proof}

Those preliminary lemmas on basic properties of the metric and the weights will be used in the next Section.
We continue by linking spaces of symbols $S(M,g)$ for weights $M$ admissible for $g$ and classical spaces of symbols $S^m_{\rho, \delta}$.

\begin{lemma}[Embeddings]
	\label{4.lemma.embedding}
	For all $m\in\R$, the following embedding holds
	\be
		 S^{m}_{1,0} \, \subset \, S\left( \langle \cdot \rangle^m , g \right) .
	\ee

	\noindent Let $M$ be an admissible weight satisfying $M(x,\xi) \leq \langle \xi \rangle^m$ for all $(x,\xi) \in \R \times \R$ for some $m\in\R$. Then the following embedding holds
	\be
		S(M,g) \, \subset \, S^m_{1, c/2} .
	\ee

\end{lemma}

\noindent The proof is postponed in the Appendix.
%
%
%
%

\subsection{pseudo-differential calculus}

We use here the Weyl quantization, which we recall
$$
	{\rm op}(a)u(x) = {\rm op}_{1/2}(a)u(x) = \int e^{2\pi i (x-y)\cdot \xi} a\left(\frac{x+y}{2}, \xi\right) u(y) dy d\xi .
$$

We recall the algebra property of general classes of symbols $S(M,g)$.
Let $M_1$ and $M_2$ be both admissible weights for the metric $g$.
\begin{lemma}
	\label{4.lemma.algebra}
	For any $p_j \in S(M_j,g)$ with $j=1,2$, there holds
	$$
		p_1 p_2 \in S(M_1M_2,g) .
	$$
\end{lemma}

\noindent The proof is straightforward, using Leibniz formula and Definition \ref{4.defi.class.symbols}.

\medskip

We now state Theorem {\rm 2.3.7} in {\rm \cite{lerner2011metrics}}, concerning the composition of operators with symbols in $S(M,g)$.
For two symbols $p_1$ and $p_2$, we denote $p_1 \sharp p_2$ the symbol satisfying
$$
	{\rm op}(p_1) {\rm op}(p_2) = {\rm op}(p_1 \sharp p_2) .
$$

\noindent We denote also
\be
	\label{4.def.lambda}
	\lambda(t,x,\xi) = b^{-1}(t,x,\xi) \langle \xi \rangle.
\ee

\begin{lemma}[Composition]
	\label{4.lemma.composition}
	Let $M_1$ and $M_2$ two admissible weights for $g$, and $p_j \in S(M_j,g)$.
	Then for all $\nu$ in $\N$ there holds
	\be
		\label{composition.general}
		p_1 \sharp p_2 - \left( \sum_{ 0 \leq k < \nu } 2^{-k} \sum_{|\a|+|\beta|=k} \frac{(-i)^{|\beta|}}{\a!\beta!} \d_{\xi}^{\beta}\d_x^{\a}p_1 \d_{\xi}^{\a} \d_{x}^{\beta}p_2 \right) \, \in S(M_1M_2\lambda^{-\nu}, g)
	\ee

	\noindent where $\lambda$ is defined by \eqref{4.def.lambda}.
	In particular, there holds
	\be
		\label{composition.crochet}
		p_1 \sharp p_2 - \frac{1}{2i} \{p_1 , p_2 \} \, \in S(M_1M_2\lambda^{-2}, g)
	\ee

	\noindent where $\{ \cdot, \cdot \}$ denotes the usual Poisson bracket on $\R_{x} \times \R_{\xi}$.
	About commutators:
	\be
		[\op(p_1) , \op(p_2)] \in \op \left( S \left( M_1 M_2 \lambda^{-2} , g \right) \right) .
	\ee
\end{lemma}

In the previous Lemma, we see that the powers of the symbol $\lambda$ act as a gradation for the remainder term in composition of operators.
In the case of usual flat metrics on the phase space, the symbol $\lambda$ is simply equal to $\langle \xi \rangle$, and the previous Lemma reads (for instance) as $p_1 \sharp p_2 - p_1p_2 \in S^{m_1+m_2-1}_{1,0}$ for $p_1$ and $p_2$ in $S^{m_j}_{1,0}$.

We pursue by giving a result on the inversion of $\op(b)$ up to any low order remainder term.
\begin{lemma}[Inversion of $\op(b)$]
	\label{lemma.inversion}
	For any $\nu \in \N$, there is a symbol $c_{\nu}$ in $S(b^{-1} , g)$ such that
	\be
		\op(b) \op(c_{\nu}) - {\rm Id} \in \op \left( S(\lambda^{-(\nu+1)} , g) \right) .
	\ee
\end{lemma}

\begin{proof}
	Our aim is to solve the equation
	$$
		b \sharp c_{\nu} = 1 + S(\lambda^{-(\nu+1)} , g) .
	$$

	\noindent We proceed by induction on $\nu$, using equality \eqref{composition.general}.
	Denoting
	$$
		\omega_{k}(p_1 , p_2) = 2^{-k} \sum_{|\a|+|\beta|=k} \frac{(-i)^{|\beta|}}{\a!\beta!} \d_{\xi}^{\beta}\d_x^{\a}p_1 \d_{\xi}^{\a} \d_{x}^{\beta}p_2
	$$

	\noindent equality \eqref{composition.general} states
	$$
		p_1 \sharp p_2 = \sum_{ 0 \leq k < \nu } \omega_k(p_1 , p_2) + S(M_1M_2\lambda^{-\nu}, g) .
	$$

	\noindent Note that in particular
	$$
		\omega_0(p_1 , p_2) = p_1 p_2 \quad , \quad \omega_1(p_1 , p_2) = -\frac{i}{2}\{ p_1 , p_2 \}.
	$$

	\noindent For $\nu = 0$, there holds
	$$
		b \sharp c_{0} = bc_0 + S(\lambda^{-1} , g)
	$$

	\noindent so that
	$$
		c_{0} = b^{-1} .
	$$

	Let $\nu \geq 0$, and assume $c_{\nu - 1} \in S(\lambda^{-(\nu-1)}, g) $ solves
	$$
		b \sharp c_{\nu - 1} = 1 + S(\lambda^{-\nu}, g) .
	$$
	Then on one side
	\begin{eqnarray*}
		b \sharp \left( c_{\nu} - c_{\nu-1} \right)
		& = &
		\sum_{ 0 \leq k < \nu+1 } \omega_k(b , c_{\nu} - c_{\nu-1}) + S(\lambda^{-(\nu+1)}, g) \\
		& = &
		b ( c_{\nu} - c_{\nu-1} ) + S(\lambda^{-(\nu+1)}, g)
	\end{eqnarray*}

	\noindent as $c_{\nu} - c_{\nu-1} \in S(\lambda^{-\nu}, g)$, and on the other side thanks to the equation there holds
	$$
		b \sharp \left( c_{\nu} - c_{\nu-1} \right) = \left( 1 - b \sharp c_{\nu-1} \right) + S(\lambda^{-(\nu+1)}, g)
	$$

	\noindent which leads to
	$$
		c_{\nu} - c_{\nu-1} = b^{-1}\left( 1 - b \sharp c_{\nu-1} \right)
	$$

	\noindent which is in $S(b^{-1},g)$.
\end{proof}

Finally, we recall Theorem 2.5.1 of \cite{lerner2011metrics}.
\begin{lemma}[Action]
	\label{4.lemma.action}
	Let $p$ be in $S(1,g)$.
	Then $\op(p)$ acts continuously on $L^2$.
\end{lemma}


\subsection{Energy estimate}


In Section \ref{4.subsection.derivative.energy}, we observed cancellations in $\d_t E$.
The next step is to bound the remainder terms in ${\rm E_2}$, ${\rm E_3}$ and ${\rm E_4}$ by a fraction of the negative term ${\rm E_1}$ .
This is done thanks to the properties of the pseudo-differential calculus described in Appendix \ref{4.subsection.metrics} and Lemma \ref{4.lemma.embedding} ; by choice of the exponent $c>0$ of the correction term $\langle \xi \rangle^{-c}$ which appears in the definition \eqref{def.a.natural} of $a_{\natural}$ ; and by a lower bound on the Gevrey index $\sigma$.

\subsubsection{Estimate of ${\rm E_{2}}$}

The term ${\rm E_{2}}$, defined in \eqref{4.FII} is equal, thanks to the previous computations, to the sum of \eqref{4.FII.1}, \eqref{4.FII.2} and \eqref{4.FII.3}.

$\bullet$ First we focus on \eqref{4.FII.1}, using the results of Section \ref{subsubsection.E1}.
We write
\begin{eqnarray*}
	- {\rm Re}\, \langle \op(S)
	\begin{pmatrix}
		0 & 0 \\
		D^{-c} & 0
	\end{pmatrix}
	\d_{x} v, \op(S) v \rangle
	& = &
	- {\rm Re}\, \langle\op(b) D^{-c} \d_x v_1 , \op(b) v_2 \rangle \\
	& = &
	- {\rm Re}\, \langle D^{-\sigma/2} \op(b) D^{-c-\sigma/2} \d_x D^{\sigma/2} v_1 , D^{\sigma/2} \op(b) v_2 \rangle .
\end{eqnarray*}

\noindent The operator $D^{-\sigma/2} \op(b) D^{-c-\sigma/2} \d_x$ has a symbol in the class $S \left( b \langle \cdot \rangle^{1 - c - \sigma} , g \right)$, which is embedded in $S(1,g)$ as soon as
\be
	\label{constraint.error}
	1 - c/2 - \sigma \leq 0 .
\ee

\noindent If the constraint is satisfied, the operator acts continuously on $L^2$, hence
\be
	\label{4.FII.1.F.I}
	\left|\eqref{4.FII.1}\right| \lesssim {\rm E_{1}}.
\ee

$\bullet$ Second, we focus on \eqref{4.FII.2}.
Here, equality \eqref{composition.crochet} of Lemma \ref{4.lemma.composition} and cancellations of brackets $\{b,b\} = 0$ and $\{b, a_{\natural}\} = 0$ imply that
$$
	\op\left( S^2 A_{\natural} \right) - \op(S)^2 \op(A_{\natural})
	=
	\begin{pmatrix}
		0 & 0 \\
		\op\left( S(\lambda^{-2},g) \right) & 0
	\end{pmatrix} .
$$

\noindent Hence
\begin{eqnarray*}
	{\rm Re}\, \langle \left( \op\left( S^2 A_{\natural} \right) - \op(S)^2 \op(A_{\natural}) \right) \d_x v, v \rangle
	& = &
	{\rm Re}\, \langle
		\begin{pmatrix}
			0 & 0 \\
			\op\left( S(\lambda^{-2},g) \right) & 0
		\end{pmatrix}
		\d_x v, v
	\rangle \\
	& = &
	{\rm Re}\, \langle \op\left( S(\lambda^{-2},g) \right) \d_x v_1, v_2 \rangle .
\end{eqnarray*}

\noindent Next, we proceed as we did in the previous point for \eqref{4.FII.1}.
As we control $ D^{\sigma/2} \op(b) v_2$ in $L^2$ norm, we use Lemma \ref{lemma.inversion} to make appear $\op(b)$ up to a remainder in $S(\lambda^{-(\nu-1)},g)$ for $\nu$ to be chosen later.
Hence there holds
\begin{align*}
	& {\rm Re}\, \langle \op\left( S(\lambda^{-2},g) \right) \d_x v_1, v_2 \rangle \\
	& = {\rm Re}\, \langle \op(c_{\nu})\op\left( S(\lambda^{-2},g) \right) \d_x v_1, \op(b)v_2 \rangle + {\rm Re}\, \langle \op\left( S(\lambda^{-2 - \nu-1},g) \right) \d_x v_1, v_2 \rangle \\
	& = {\rm Re}\, \langle D^{-\sigma/2}\op(c_{\nu})\op\left( S(\lambda^{-2},g) \right)D^{-\sigma/2} \d_x D^{\sigma/2} v_1, D^{\sigma/2}\op(b)v_2 \rangle + R_{2,2}
\end{align*}

\noindent with
$$
	R_{2,2} = {\rm Re}\, \langle \op\left( S(\lambda^{-3 - \nu},g) \right) \d_x v_1, v_2 \rangle .
$$

\noindent As $c_{\nu}$ is in $ S(b^{-1},g)$, we get
$$
	D^{-\sigma/2}\op(c_{\nu})\op\left( S(\lambda^{-2},g) \right) D^{-\sigma/2} \d_x = \op \left( S\left(\langle \cdot \rangle^{-1 - \sigma} b,g\right) \right) .
$$

\noindent Thanks to inequality \eqref{4.upper.bound.b}, the symbol satisfies $\langle \cdot \rangle^{-1 - \sigma} b \leq 1$ which implies the boundedness in $L^2$ of the previous operator, hence
$$
	\left| {\rm Re}\, \langle D^{-\sigma/2}\op(c_{\nu})\op\left( S(\lambda^{-2},g) \right)D^{-\sigma/2} \d_x D^{\sigma/2} v_1, D^{\sigma/2}\op(b)v_2 \rangle \right|
	\lesssim
	{\rm E_{1}} .
$$

We consider now the remainder term $R_{2,2}$, writing
\be
	\label{4.local.R22}
	R_{2,2} = {\rm Re}\, \langle \op\left( S\left( \langle \cdot \rangle^{1-\sigma} \lambda^{-3 - \nu},g \right) \right) D^{\sigma/2} v_1, D^{\sigma/2} v_2 \rangle
\ee

\noindent and both definition \eqref{4.def.lambda} of $\lambda$ and inequality \eqref{4.upper.bound.b} imply that
$$
	\op\left( S\left( \langle \cdot \rangle^{1-\sigma} \lambda^{-3 - \nu},g \right) \right) \subset \op\left( S\left( \langle \cdot \rangle^{-2-\nu-\sigma + c(3 + \nu)/2} ,g \right) \right) .
$$

\noindent As soon as the constraint
\be
	\label{4.constraint.2}
	\frac{c}{2} \leq 1 - \frac{1 - \sigma}{3 + \nu}
\ee

\noindent is satisfied, operators $\op\left( S\left( \langle \cdot \rangle^{1-\sigma} \lambda^{-3-\nu},g \right) \right)$ act thus continuously on $L^2$ thanks to Lemma \ref{4.lemma.action}.
Then there holds
$$
	\left| R_{2,2} \right| \lesssim {\rm E_{1}}
$$

\noindent and finally
\be
	\label{4.FII.2.F.I}
	\left| \eqref{4.FII.2} \right| \lesssim {\rm E_{1}}.
\ee

\begin{remark}
	Constraint \eqref{4.constraint.2} is essentially technical.
	Note that, as $\nu$ goes to infinity, constraint \eqref{4.constraint.2} becomes $c/2 <1$.
	This is the uncertainty principle for the metric $g$.

%
%
\end{remark}

\smallskip

$\bullet$ Third, we focus on \eqref{4.FII.3}.
Thanks to Lemma 5.2 in \cite{morisse2016j}, the symbol $\widetilde{a}$ satisfies
\be
	\label{4.expansion.tilde.a}
	\tilde{a} - a = i\d_{x}a \, \d_{\xi} \langle \xi \rangle^{\sigma} + S^{-2(1-\sigma)}_{1,0}
\ee

\noindent and we write thus
\begin{eqnarray}
	{\rm Re}\, \langle \op(S) \op(\widetilde{A} - A) \d_{x} v, \op(S)v \rangle
	& = &
	{\rm Re}\, \langle \op(S)
		\begin{pmatrix}
			0 & 0 \\
			\op(\widetilde{a}) - a & 0
		\end{pmatrix}
	\d_{x} v, \op(S)v \rangle \nonumber \\
	& = &
	{\rm Re}\, \langle \op(b) {\rm op}(\widetilde{a} - a) \d_{x} v_1, \op(b) v_2 \rangle \nonumber \\
	& = &
	{\rm Re}\, \langle \op(b) \op\left( i\d_{x}a \, \d_{\xi} \langle \cdot \rangle^{\sigma} \right) \d_{x} v_1, \op(b) v_2 \rangle + R_{2,3} \label{4.local.R23}
\end{eqnarray}

\noindent where
$$
	R_{2,3} = {\rm Re}\, \langle \op(b) {\rm op}\left( S^{-2(1-\sigma)}_{1,0} \right) \d_{x} v_1, \op(b) v_2 \rangle.
$$

The sub-principal symbol $i\d_{x}a \, \d_{\xi} \langle \xi \rangle^{\sigma}$ is \textit{a priori} in $S^{-(1-\sigma)}_{1,0}$, which would be insufficient to counterbalance both ${\rm op}(b)$ and $\d_x$.
Indeed, by Lemma \ref{4.lemma.b.symbol} and Lemma \ref{4.lemma.embedding}, there holds $\op(b) \d_x \in \op \left( S^{1 + c/2}_{1,c/2} \right)$ \textit{versus} the straigthforward estimate $ \op\left( i\d_{x}a \, \d_{\xi} \langle \xi \rangle^{\sigma} \right) \in \op \left( S^{-(1-\sigma)}_{1,0} \right) $.
But using the Glaeser inequality for $a$ described in Lemma \ref{4.lemma.glaeser.und.a} and definition \eqref{4.def.metric} of the metric $g$, we prove that in fact
\be
	\label{4.subprincipal}
	i\d_{x}a \, \d_{\xi} \langle \xi \rangle^{\sigma} \in S(b^{-1} \langle \cdot \rangle^{\sigma-1}, g) .
\ee

\noindent Indeed for any $\a$, $\beta$ in $\N$, there holds
\begin{eqnarray*}
	\left| \d_{x}^{\a} \d_{\xi}^{\beta} \left( i\d_{x}a(t,x) \, \d_{\xi} \langle \xi \rangle^{\sigma} \right) \right| & = & \left| \d_{x}^{\a+1}a(t,x) \, \d_{\xi}^{\beta+1} \langle \xi \rangle^{\sigma} \right| \\
		& \lesssim & \left| \d_{x}^{\a+1}a(t,x) \right| \, \langle \xi \rangle^{\sigma-1-|\beta|} .
\end{eqnarray*}

\noindent The lower bound \eqref{4.lower.bound.b} for $b$ implies
$$
	\left| \d_{x}^{\a+1}a(t,x) \right| \leq \left| \d_{x}^{\a+1}a \right|_{L^{\infty}([0,T]\times B_{r}(x_0)} b^{-1 + |\a|} \quad , \quad \forall \,(t,x) \in [0,T]\times B_{r}(x_0) .
$$

\noindent For $|\a| = 0$, Glaeser's Lemma \ref{4.lemma.glaeser.und.a} and definition \eqref{4.def.b} of $b$ lead to
\begin{eqnarray*}
	\left| \d_{x}^{\a+1}a(t,x) \right| & \lesssim & \left| \d_{x}^{2}a \right|_{L^{\infty}([0,T]\times B_{r}(x_0)}^{1/2} a(t,x)^{1/2} \\
		& \lesssim & \left| \d_{x}^{2}a \right|_{L^{\infty}([0,T]\times B_{r}(x_0)}^{1/2} \left(a(t,x) + \langle \xi \rangle^{-c} \right)^{1/2} \\
		& \lesssim & \left| \d_{x}^{2}a \right|_{L^{\infty}([0,T]\times B_{r}(x_0)}^{1/2} b(t,x,\xi)^{-1+|\a|}
\end{eqnarray*}

\noindent for any $(t,x,\xi) \in [0,T] \times B_{r}(x_0) \times \R$. Thus, for any $\a\in\N$, there is $C_{\a} >0$ such that
$$
	\left| \d_{x}^{\a} \d_{\xi}^{\beta} \left( i\d_{x}a(t,x) \, \d_{\xi} \langle \xi \rangle^{\sigma} \right) \right| \lesssim C_{\a} \,b(t,x,\xi)^{-1 + |\a|} \, \langle \xi \rangle^{\sigma-1-|\beta|} \quad , \quad \forall \, (t,x,\xi) \in [0,T] \times B_{r}(x_0) \times \R
$$

\noindent combining both cases, hence the proof of \eqref{4.subprincipal}.

For the first term in the right-hand side of \eqref{4.local.R23}, we follow the same path as in the above treatment of \eqref{4.FII.1} and \eqref{4.FII.2}, writing
\begin{align*}
	& {\rm Re}\, \langle \op(b) \op\left( i\d_{x}a \, \d_{\xi} \langle \cdot \rangle^{\sigma} \right) \d_{x} v_1, \op(b) v_2 \rangle \\
	& =
	{\rm Re}\, \langle D^{-\sigma/2} \op(b) \op\left( i\d_{x}a \, \d_{\xi} \langle \cdot \rangle^{\sigma} \right) D^{-\sigma/2}\d_{x} \, D^{\sigma/2} v_1, D^{\sigma/2} \op(b) v_2 \rangle \\
	& =
	{\rm Re}\, \langle \op\left( S\left( 1 , g \right) \right) \, D^{\sigma/2} v_1, D^{\sigma/2} \op(b) v_2 \rangle .
\end{align*}

\noindent Hence, by Lemma \ref{4.lemma.action},
$$
	\left| {\rm Re}\, \langle \op(b) \op\left( i\d_{x}a \, \d_{\xi} \langle \cdot \rangle^{\sigma} \right) \d_{x} v_1, \op(b) v_2 \rangle \right| \lesssim {\rm E_{1}} .
$$

For the remainder term $R_{2,3}$, there holds
$$
	{\rm Re}\, \langle \op(b) \op\left( S_{1,0}^{-2(1-\sigma)} \right) \d_{x} v_1, \op(b) v_2 \rangle
	=
	{\rm Re}\, \langle \op\left( S \left( b \langle \cdot \rangle^{-(1-\sigma)} , g \right) \right) D^{\sigma/2} v_1, D^{\sigma/2} \op(b) v_2 \rangle .
$$

\noindent Thanks to inequality \eqref{4.upper.bound.b} on $b$, we prove $b \langle \cdot \rangle^{-(1-\sigma)} \leq \langle \cdot \rangle^{c/2 + \sigma-1}$ which implies
$$
	\op\left( S \left( b \langle \cdot \rangle^{-(1-\sigma)} , g \right) \right) \subset \op\left( S \left( \langle \cdot \rangle^{c/2 + \sigma-1} , g \right) \right) .
$$

\noindent Hence, as soon as
\be
	\label{4.constraint.3}
	c/2 + \sigma - 1 \leq 0
\ee

\noindent holds, operators $\op\left( S \left( b \langle \cdot \rangle^{-(1-\sigma)} , g \right) \right)$ act on $L^2$ thanks to Lemma \ref{4.lemma.action}, thus
\be
	\label{4.FII.3.F.I}
	\left| \eqref{4.FII.3} \right| \lesssim {\rm E_{1}}
\ee

\noindent using again Lemma \ref{4.lemma.action}.

Putting together estimates \eqref{4.FII.1.F.I}, \eqref{4.FII.2.F.I} and \eqref{4.FII.3.F.I}, there is a constant $C_{{\rm 2}} >0$ such that
\be
	\label{4.def.CII}
	\left| {\rm E_{2}} \right| \leq C_{{\rm 2}} \, {\rm E_{1}}.
\ee

\begin{remark}
	The discussion before estimate \eqref{4.subprincipal} on the subprincipal symbol of $\tilde{a}$ stresses out the importance of a careful study of subprincipal symbols when dealing with weakly hyperbolic systems.
	It echoes the considerations and computations which lead to equality \eqref{subprincipal.dtb} for $\d_t b$.
\end{remark}

\subsubsection{Estimate of ${\rm E_{3}}$}

We proceed as before, focusing first on the first term of the right-hand side of inequality \eqref{4.FIII.1}.
Using Lemma \ref{lemma.inversion}, we write
\begin{align*}
	& {\rm Re}\, \langle \op \left( S(b^5 \langle \cdot \rangle^{-c-1} , g) \right)  v_2, \op(b) v_2 \rangle \\
	& =
	{\rm Re}\, \langle D^{-\sigma/2} \op \left( S(b^5 \langle \cdot \rangle^{-c-1} , g) \right) \op\left(c_{\nu}\right) D^{-\sigma/2}\, D^{\sigma/2} \op(b) v_2, D^{\sigma/2} \op(b) v_2 \rangle \\
	&
	\quad + {\rm Re}\, \langle D^{-\sigma/2} \op \left( S(b^5 \langle \cdot \rangle^{-c-1} , g) \right) \op \left( S \left( \lambda^{-\nu-1} , g \right) \right) v_2, D^{\sigma/2}\op(b) v_2 \rangle \\
	& =
	{\rm Re}\, \langle \op \left( S(b^4 \langle \cdot \rangle^{-\sigma-c-1} , g) \right) \, D^{\sigma/2} \op(b) v_2, D^{\sigma/2} \op(b) v_2 \rangle \\
	&
	\quad + {\rm Re}\, \langle \op\left( S\left( b^5 \langle \cdot \rangle^{-\sigma-c-1} \lambda^{-\nu-1} , g\right) \right) D^{\sigma/2} v_2, D^{\sigma/2}\op(b) v_2 \rangle .
\end{align*}

\noindent Following the scheme developed in the previous points, we use inequality \eqref{4.lower.bound.b} to get both bounds
$$
	b^4 \langle \cdot \rangle^{-\sigma-c-1}
	\leq
	\langle \cdot \rangle^{c - \sigma -1}
$$

\noindent and
$$
	b^5 \langle \cdot \rangle^{-\sigma-c-1} \lambda^{-\nu-1}
	=
	b^{6+\nu} \langle \cdot \rangle^{-\sigma-c-2-\nu}
	\leq
	\langle \cdot \rangle^{c(4+\nu)/2-\sigma-\nu-2} .
$$

\noindent To use Lemma \ref{4.lemma.action} for $L^2$-boundedness of the associated operators, we need constraints
\be
	\label{constraint.dtb}
	c - \sigma -1
	\leq
	0
\ee

\noindent and
\be
	\label{constraint.dtb.remainder}
	c(4+\nu)/2-\sigma-\nu-2
	\leq
	0.
\ee

\noindent As soon as both constraints are satisfied, by inequality \eqref{4.FIII.1}, there is $C_{{\rm 3}} >0$ such that
\be
	\label{4.def.CIII}
	\left| {\rm E_{3}} \right| \leq C_{{\rm 3}} \, {\rm E_{1}} .
\ee

\subsubsection{Estimate of ${\rm E_{4}}$}

We consider now the non-linear, 0th order term ${\rm E_{4}}$.
Without any additional assumption on the (matrix) structure of the non-linearity, the control of the source term may lead to another constraint linking $c$ and $\sigma$.
In particular, however we control $D^{\sigma/2} \op(b) v_2$ in an $L^2$ norm, the term ${\rm op}(b) \left( F(u)^{(\tau)} v \right)_2$ in ${\rm E_{4}}$ may not be controlled in the same way.
We thus have to bound the operator $\op(b)$ using $D^{-\sigma}$.

We write first, as before,
\begin{align}
	& {\rm Re}\, \langle {\rm op}(b) \left( F(u)^{(\tau)} v \right)_2, {\rm op}(b) v_2 \rangle + {\rm Re} \, \langle \left( F(u)^{(\tau)} v \right)_1 , v_1 \rangle \nonumber \\
	& = {\rm Re}\, \langle D^{-\sigma/2} {\rm op}(b)D^{-\sigma/2} \, D^{\sigma/2}\left( F(u)^{(\tau)} v \right)_2, D^{\sigma/2} {\rm op}(b) v_2 \rangle + {\rm Re} \, \langle \left( F(u)^{(\tau)} v \right)_1 , v_1 \rangle \label{4.local.E4}
\end{align}

\noindent Next, by Lemma \ref{4.lemma.embedding} and Lemma \ref{4.lemma.composition}, there holds
$$
	D^{-\sigma/2}\op(b) D^{-\sigma/2} \in \op \left( S\left( b \langle \cdot \rangle^{-\sigma}, g \right) \right) .
$$

\noindent Thanks to the upper bound \eqref{4.upper.bound.b} for $b$, we get
$$
	\op \left( S\left( b \langle \cdot \rangle^{-\sigma}, g \right) \right) \subset \op \left( S\left( \langle \cdot \rangle^{c/2-\sigma}, g \right) \right) .
$$

\noindent As soon as the constraint
\be
	\label{constraint.nonlinear}
	c/2 - \sigma \leq 0
\ee

\noindent is satisfied, operators $\op \left( S\left( \langle \cdot \rangle^{c/2-\sigma}, g \right) \right)$ act continuously on $L^2$ thanks to Lemma \ref{4.lemma.action}.
This implies
$$
	\left| D^{-\sigma/2} {\rm op}(b)D^{-\sigma/2} \, D^{\sigma/2}\left( F(u)^{(\tau)} v \right)_2 \right|_{L^2} \lesssim \left| D^{\sigma/2} \left( F(u)^{(\tau)} v \right)_2  \right|_{L^2} .
$$

\begin{remark}
	\label{remark.hypo.nonlinear}
	Constraint \eqref{constraint.nonlinear} comes from control of the term $\op(b) \left( F(u)^{(\tau)} v \right)_2$, which decomposes into
	$$
		\op(b) \left( F(u)^{(\tau)} v \right)_2 = \op(b) F(u)_{21}^{(\tau)} v_1  + \op(b) F(u)_{22}^{(\tau)} v_2 .
	$$

	\noindent At a first level of approximation, there holds
	$$
		\op(b) \left( F(u)^{(\tau)} v \right)_2 \approx F(u)_{21}^{(\tau)} \op(b) v_1  + F(u)_{22}^{(\tau)} \op(b) v_2 .
	$$

	\noindent The second term of the right-hand side may be controlled directly by the term ${\rm E_{1}}$, but not the first term, as $\op(b) v_1$ is not a priori controlled in a $H^{\sigma/2}$ norm.

	Adding the structural assumption that $F(u)_{21} \equiv 0$ may then help loosen the constraint \eqref{constraint.nonlinear}, and in the end the lower bound on the Gevrey index.
	This is typical of weakly hyperbolic systems: a perturbation by a lower order term may induce a Gevrey loss of regularity.
	A careful analysis of subprincipal symbol involving the approximated symbol $a_{\natural}$ is thus of great importance.
\end{remark}

The control of non-linearity is made thanks to the property of algebra of Gevrey spaces, and the analytical structure of $F(u)$.
As $u$ is in ${\rm \mathcal{G}}^{\sigma}_{\tau}$, Assumption \ref{hypo.reg.F} and the property that $H^{\sigma/2}{\rm \mathcal{G}}^{\sigma}_{\tau}$ is an algebra thanks to Remark 3 in \cite{morisse2016j}, Proposition 3.2 therein implies that $F(u)^{(\tau)}$ acts continuously in $H^{\sigma/2}$, hence
$$
	\left| D^{\sigma/2} \left( F(u)^{(\tau)} v \right)_2  \right|_{L^2} \lesssim \left\| F(u)^{(\tau)} \right\|_{\mathcal{L}(H^{\sigma/2})} \left| D^{\sigma/2} v\right|_{L^2}.
$$

\noindent Using Cauchy-Schwarz' inequality to get an estimate of \eqref{4.local.E4}, there holds
\begin{align*}
	& \left| \eqref{4.local.E4} \right| \\
	& \lesssim \left| D^{-\sigma/2} {\rm op}(b) \left( F(u)^{(\tau)} v \right)_2\right|_{L^2} \, \left| {\rm op}(b) D^{\sigma/2} v_2 \right|_{L^2} + \left| \left( F(u)^{(\tau)} v \right)_1 \right|_{L^2} \, \left| v_1 \right|_{L^2} \\
	& \lesssim \left\| F(u)^{(\tau)} \right\|_{\mathcal{L}(H^{\sigma/2})} \left| D^{\sigma/2} v\right|_{L^2} \, \left| {\rm op}(b) D^{\sigma/2} v_2 \right|_{L^2} + \left\| F(u)^{(\tau)} \right\|_{\mathcal{L}(L^2)} \, \left| v_1 \right|_{L^2}^2 \\
	& \lesssim \left( \left\| F(u)^{(\tau)} \right\|_{\mathcal{L}(H^{\sigma/2})} + \left\| F(u)^{(\tau)} \right\|_{\mathcal{L}(L^2)} \right){\rm E_{1}} .
\end{align*}

We conclude by
\be
	\label{4.def.CIV}
	\left|{\rm E_{4}} \right| \leq C_{{\rm 4}} \, {\rm E_{1}}
\ee

\noindent for some $C_{{\rm 4}}>0$ depending essentially on $\left\| F(u)^{(\tau)} \right\|_{\mathcal{L}(H^{\sigma/2})} + \left\| F(u)^{(\tau)} \right\|_{\mathcal{L}(L^2)}$.

\subsubsection{Conclusion}

In order to complete the proof of Theorem \ref{4.theo}, we put together the different constraints between $c$ and $\sigma$ that appear in the estimates of the energy.
First, combining constraint \eqref{constraint.error} and constraint \eqref{4.constraint.3}, there holds
$$
	\frac{c}{2} = 1 - \sigma .
$$

\noindent This equality between the parameter $c$, used to regularize the weakness in the hyperbolicity of the system, and the Gevrey index $\sigma$ already appeared in the seminal paper \cite{colombini1983well}.

Next, we gather constraints \eqref{4.constraint.2}, \eqref{constraint.dtb}, \eqref{constraint.dtb.remainder} and \eqref{constraint.nonlinear}.
The last constraint implies immediately the expected lower bound for the Gevrey index
\be
	\label{4.sigma.demi}
	\sigma \geq 1/2 .
\ee

\noindent The three other constraints are weaker, hence do not interfere in the lower bound.
However, as soon as constraint \eqref{constraint.nonlinear} breaks down (see Remark \ref{remark.hypo.nonlinear}), the lower bound for the Gevrey index is better.
Indeed, constraint \eqref{constraint.dtb} implies the inequality $ 1 - \sigma \leq \frac{1}{2}(1 + \sigma )$, equivalent to
\be
	\sigma \geq \frac{1}{3} .
\ee

\noindent Note that both constraints \eqref{4.constraint.2} and \eqref{constraint.dtb.remainder} are equivalent, as $\nu$ tends to infinity, to $c/2 \leq 1$ -- which is the limitation of $c$ imposed by the uncertainty principle for the metric $g$.

\smallbreak

We prove Theorem \ref{4.theo} by taking ${\bm \tau} > C_{{\rm 2}} + C_{{\rm 3}} + C_{{\rm 4}}$, where the constants are defined respectively in \eqref{4.def.CII}, \eqref{4.def.CIII} and \eqref{4.def.CIV}.

\section{Appendices: two lemmas of real analysis and metrics in the phase space}


\subsection{Glaeser-type inequalities}
\label{4.subsection.glaeser}


We start by recalling the Fa\`a di Bruno formula on iterated derivatives of composition of functions:

\begin{lemma}[Fa\`a di Bruno formula]
	\label{4.lemma.faa}
	Let $f: \R^{d}\times\R^{d} \to \R$ and $g:\R \to \R$ be two $C^{\infty}$ functions. Then
	\be
		\forall\, \alpha , \beta \in \mathbb{N}^d
		\quad , \quad
		\frac{ \d_{x}^{\a}\d_{\xi}^{\beta}(g\circ f) }{ \a ! \beta ! } = \sum_{1 \leq k \leq |\a + \beta|} \frac{ g^{(k)}\circ f}{k!} \sum_{ (\a_1,\beta_1) + \cdots + (\a_k,\beta_k) = (\a,\beta) \atop (\a_j,\beta_j)\neq(0,0)} \prod_{1\leq j \leq k} \frac{ \d_{x}^{\a_j}\d_{\xi}^{\beta_j}f }{ \a_j ! \beta_j ! } .
	\ee
\end{lemma}

We recall that for a $d$-tuple $\a_j = (\a_j(1), \ldots,\a_j(d))$, we denote $\a_j! = \prod_{1\leq p \leq d} \a_j(p)!$, and $\d_{x}^{\a_j}$ means $\d_{x_1}^{\a_j(1)} \circ \cdots \circ \d_{x_d}^{\a_j(d)}$.
For further use, we denote
\be
	N(\a,k) = \left| \left\{ (\a_1, \ldots , \a_k) \,\Big| \, \a_1 + \cdots + \a_k = \a \, , \, \a_j \geq 1 \right\} \right|.
\ee

\noindent By combinatorial arguments, for all $\alpha \in \N^{d}$ and $k\geq 1$ there holds
$$
	N(\a , k) =  \prod_{1\leq j \leq d} \binom{\a(j) - 1}{k-1} .
$$

\noindent By putting $f(y) = y^{n}$ and $g(x) = e^{x}$ in the Fa\`a di Bruno formula, we obtain
$$
	n^{\a} = \sum_{1 \leq k \leq |\a|} \binom{n}{k} \sum_{\a_1 + \cdots + \a_k = \a \atop \a_j \geq 1} \binom{\a}{\a_1 , \ldots , \a_k} .
$$

\smallskip

Next we recall the classical Glaeser inequality (see \cite{glaeser1963racine}):
\begin{lemma}[Global Glaeser inequality]
	\label{4.lemma.glaeser.global}
	Let $f:\R^{n} \to \R$ be a non negative $C^2$ function, such that $\d_x^2f$ is bounded.
	Then
	\be
		\label{4.ineq.glaeser.global}
		\forall x\in\R^{n} \;,\quad |\d_{x}f(x)|^2 \leq 2|\d_x^2f|_{L^{\infty}(\R^{n})} f(x) .
	\ee
\end{lemma}

The local result (inequality holds at any point) comes from a global assumption on $f$ (non negativity of $f$, boundedness of $\d_x^2f$).
The constant $2|\d_x^2f|_{L^{\infty}(\R^{n})}$ is optimal.
The proof of the Lemma is classical, and is based on the integral Taylor expansion formula.

Local versions of the previous statement, that is with assumptions valid only in an open set of $\R^{n}$, also exist.
For any $x_0 \in\R^d$ and $r>0$, we denote
$$
	B_{r}(x_0) = \left\{ x \in \R^d \,:\, |x - x_0| < r \right\} .
$$

\noindent In all the following, we consider $f : B_r(x_0) \to \R$ a nonnegative, $C^2$ function.
We give first a sharp version of a local Glaeser's inequality, used in the present paper.
\begin{lemma}[Sharp local Glaeser inequality]
	\label{4.lemma.sharp.local.glaeser}
	$ {}^{} $

	Assuming that
	$$
		\min_{x\in \overline{B}_{r}(x_0)} f(x) >0
	$$

	\noindent then, for any $p>0$ and any $r' < r$, there holds
	\be
		\label{local.trucmuche}
		\forall x\in \overline{B}_{r'}(x_0) \;,\quad
		|\d_{x}f(x)|^{p} \leq \frac{\left(|\d_{x}f|_{L^{\infty}(\overline{B}_{r'}(x_0))}\right)^{p} }{\min_{\overline{B}_{r'}(x_0)} f} f(x) .
	\ee
\end{lemma}

\begin{proof}
	The inequality \eqref{local.trucmuche} is straightforward as there holds both
	$$
		\forall x\in \overline{B}_{r'}(x_0) \;,\quad \min_{\overline{B}_{r'}(x_0)} f \leq f(x)
	$$

	\noindent and
	$$
		 \forall x\in \overline{B}_{r'}(x_0) \;,\quad |\d_{x}f(x)|^{p} \leq \left(|\d_{x}f|_{L^{\infty}(\overline{B}_{r'}(x_0))}\right)^{p} .
	$$
\end{proof}

\begin{remark}
	Note that in this case, the Glaeser constant does not depend a priori of the $L^{\infty}$ norm of the second order derivatives of $f$.
	We may indeed think of polynomials of degree $2$ which are locally bounded from below by a positive constant and have a positive discriminant.
\end{remark}

Using Lemma \ref{4.lemma.sharp.local.glaeser}, we prove here Lemmas \ref{4.lemma.glaeser.und.a} and \ref{4.lemma.derivatives.b}.
\begin{lemma}[Glaeser inequality for $a$]
	Under Assumption {\rm \ref{4.hypo.A}}, there is a neighborhood $[0,T]\times B_{r}(x_0)$ of $(0,x_0)\in\R_{t}\times\R_{x}$ and a constant $C_{T,r}>0$ for which there holds
	\be
		\forall \,(t,x)\in [0,T]\times B_{r}(x_0) \quad , \quad \left( \d_{x}a(t,x) \right)^2 \leq C_{T,r} \,a(t,x) .
	\ee
\end{lemma}

\begin{proof}[Proof of Lemma {\rm \ref{4.lemma.glaeser.und.a}}]
	Thanks to Assumption {\rm \ref{4.hypo.A}}, there holds
	$$
		\left( \d_{x} a \right)^2 = a(t,x) \left( \frac{4x^2}{t+x^2}e + (t+x^2) \frac{ \left( \d_{x}e \right)^2 }{e} + 4x \d_{x}e \right)
	$$

	\noindent and the term $\left( \frac{4x^2}{t+x^2}e + (t+x^2) \frac{ \left( \d_{x}e \right)^2 }{e} + 4x \d_{x}e \right) $ is locally bounded thanks to Lemma \ref{4.lemma.sharp.local.glaeser}.

\end{proof}

\medskip

\begin{lemma}[Derivatives of the symbol $b$]
	We recall first definition \eqref{4.def.b} of $b$:
	$$
		b(t,x,\xi) = \left( a(t,x) + \langle \xi \rangle^{-c} \right)^{-1/2} .
	$$

	\noindent There is a bounded sequence of constants $C_{\a,\beta}>0$ for which there holds
	\be
		\forall \,(\a,\beta) \in \N \times \N \;, \quad
		|\d_{x}^{\a} \d_{\xi}^{\beta} b(t,x,\xi)| \leq C_{\a,\beta} \widetilde{R}^{|\a|} 2^{|\beta|} \, \a!^{s} \beta! \, b(t,x,\xi)\,b(t,x,\xi)^{|\a|} \,\langle\xi\rangle^{-|\beta|}
	\ee

	\noindent for all $(t,x)$ in $[0,T]\times B_{r}(x_0)$ and $\xi$ in $\R$.
	We introduce the standard notation $s=1/\sigma$.
	And $\widetilde{R}$ satisfies
	\be
		\widetilde{R} = R(1 + |a|_{s,R} ) > R .
	\ee
\end{lemma}

\begin{proof}[Proof of Lemma {\rm \ref{4.lemma.derivatives.b}}]

	By the Fa\`a di Bruno formula (Lemma \ref{4.lemma.faa}) on iterated derivatives of composition of functions, using the fact that $\d_x^{\a} \d_{\xi}^{\beta} a_{\natural} \equiv 0$ as soon as $|\a| >0$ and $|\beta| >0$, we deduce
	\begin{align*}
		& \frac{1}{\a!\beta!} \,\d_{x}^{\a}\d_{\xi}^{\beta} \left( a_{\natural}^{-1/2} \right) \\
		& =
		\sum_{1 \leq k \leq |\a| \atop 1 \leq k' \leq |\beta|} \frac{ c_{k} c_{k'} }{k! k'!} \,a_{\natural}^{-1/2-k-k'} \left(\sum_{\a_1 + \cdots + \a_k = \a \atop \a_j \geq 1} \prod_{j=1}^{k} \frac{1}{\a_j!} \d_{x}^{\a_j} a \right) \times \left(\sum_{\beta_1 + \cdots + \beta_{k'} = \beta \atop \beta_j \geq 1 } \prod_{j=1}^{k'} \frac{1}{\beta_j!} \d_{\xi}^{\beta_j} \langle \cdot \rangle^{-c} \right)
	\end{align*}

	\noindent where coefficients $c_{k+k'}$ are defined by $ \left( y^{-1/2}\right)^{(k)} = c_k y^{-1/2 - k} $.

	Next, there holds
	$$
		\left| \frac{1}{\beta_j!}\d_{\xi}^{\beta_j} \langle \xi \rangle^{-c} \right| \leq \langle \xi \rangle^{ -c - |\beta_j|}
	$$

	\noindent as $c \leq 2$, hence
	$$
		\left| \sum_{\beta_1 + \cdots + \beta_{k'} = \beta \atop \beta_j \geq 1} \prod_{j=1}^{k'} \frac{1}{\beta_j!} \d_{\xi}^{\beta_j} \langle \cdot \rangle^{-c} \right| \leq N(\beta, k') \, \langle \cdot \rangle^{-k'c - |\beta|}
	$$

	\noindent where we denote
	$$
		N(\beta, k') = \left| \big\{ (\beta_1 , \ldots , \beta_{k'}) \, |\, \beta_1 + \cdots + \beta_{k'} = \beta \, , \, \beta_j \geq 1 \big\} \right| .
	$$

	\noindent Thanks to the the bound \eqref{4.upper.bound.b}, there holds $ a_{\natural}^{-1} \leq \langle \cdot \rangle^{c}$, hence
	$$
		a_{\natural}^{-k'} \left| \sum_{\beta_1 + \cdots + \beta_{k'} = \beta \atop \beta_j \geq 1} \prod_{j=1}^{k'} \frac{1}{\beta_j!} \d_{\xi}^{\beta_j} \langle \cdot \rangle^{-c} \right| \leq N(\beta,k') \,\langle \cdot \rangle^{ - |\beta|} .
	$$

	\noindent We focus now on the sum
	$$
		\sum_{\a_1 + \cdots + \a_k = \a \atop \a_j \geq 1} \prod_{j=1}^{k} \frac{1}{\a_j!} \d_{x}^{\a_j} a.
	$$

	\noindent If $|\a_j| = 1$, we may use Lemma \ref{4.lemma.glaeser.und.a} to bound $\d_{x}^{\a_j}$.
	We introduce then
	$$
		I_1(\a_1, \ldots, \a_k) = \{ j : |\a_j| = 1 \}
	$$

	\noindent and there holds
	$$
		\forall \, j \in I_1 \;, \quad |\d_{x}^{\a_j} a| \leq \left( C_{T,r}^{1/2} \,a_{\natural}^{1/2} \right)^{\a_j}
	$$

	\noindent thus
	\begin{eqnarray*}
		\left| \sum_{\a_1 + \cdots + \a_k = \a \atop \a_j \geq 1} \prod_{j=1}^{k} \frac{1}{\a_j!} \d_{x}^{\a_j} a \right|
		& \leq &
		\sum_{\a_1 + \cdots + \a_k = \a \atop \a_j \geq 1} \,\prod_{j\in I_1} \frac{1}{\a_j!} |\d_{x}^{\a_j} a| \, \prod_{j \notin I_1} \frac{1}{\a_j!} |\d_{x}^{\a_j} a| \\
		& \leq &
		\sum_{\a_1 + \cdots + \a_k = \a \atop \a_j \geq 1} \,\prod_{j\in I_1} \frac{1}{\a_j!} \left( C_{T,r}^{1/2} \,a_{\natural}^{1/2} \right)^{\a_j} \, \prod_{j \notin I_1} \frac{1}{\a_j!} |\d_{x}^{\a_j} a| .
	\end{eqnarray*}

	\noindent For indices not in $I_1$, that is for $|\a_j| \geq 2$, we use the fact that $a$ is in $G^{s}_{R}$, hence
	\begin{eqnarray*}
		\left| \sum_{\a_1 + \cdots + \a_k = \a \atop \a_j \geq 1} \,\prod_{j=1}^{k} \frac{1}{\a_j!} \d_{x}^{\a_j} a \right|
		& \leq &
		\sum_{\a_1 + \cdots + \a_k = \a \atop \a_j \geq 1} \prod_{j\in I_1} \frac{1}{\a_j!} \left( C_{T,r}^{1/2} \,a_{\natural}^{1/2} \right)^{\a_j} \, \prod_{j \notin I_1} \frac{1}{\a_j!} |a|_{s,R} \, R^{|\a_j|} \a_j!^{s} .
	\end{eqnarray*}

	\noindent As the $k$-tuple $(\a_1,\ldots,\a_k)$ satisfies $\a_1 + \cdots + \a_k = \a$, there holds $ |\a_1| + \cdots + |\a_k| = |\a| $ hence
	\begin{eqnarray*}
		|I_1(\a_1, \ldots, \a_k)|
		& = &
		|\a| - \sum_{j \notin I_1} |\a_j| \\
		& \leq & |\a| - 2( k - |I_1|)
	\end{eqnarray*}

	\noindent which leads to $|I_1| \geq 2k - |\a|$. As $a_{\natural} \leq 1$, we get
	\begin{eqnarray*}
		\left| \sum_{\a_1 + \cdots + \a_k = \a \atop \a_j \geq 1} \,\prod_{j=1}^{k} \frac{1}{\a_j!} \d_{x}^{\a_j} a \right|
		& \leq &
		a_{\natural}^{k - |\a|/2} \, \sum_{\a_1 + \cdots + \a_k = \a \atop \a_j \geq 1} \prod_{j\in I_1} \frac{1}{\a_j!} C_{T,r}^{1/2} \, \prod_{j \notin I_1} \frac{1}{\a_j!} \a_j!^{s}|a|_{s,R} \, R^{|\a_j|}  .
	\end{eqnarray*}

	\noindent We need then to compare $C_{T,r}^{1/2}$ with $|a|_{s,R} R$.
	Up to shrinking $T$, we may assume that
	\be
		C_{T,r}^{1/2} \leq |a|_{s,R} R.
	\ee

	\noindent There holds
	\begin{eqnarray*}
		\left| \sum_{\a_1 + \cdots + \a_k = \a \atop \a_j \geq 1} \prod_{j=1}^{k} \frac{1}{\a_j!} \d_{x}^{\a_j} a \right|
		& \leq &
		a_{\natural}^{k - |\a|/2} |a|_{s,R}^{k} R^{|\a|} \, \sum_{\a_1 + \cdots + \a_k = \a \atop \a_j \geq 1} \prod_{j=1}^{k}\frac{1}{\a_j!} \a_j!^{s}  \\
		& \leq &
		a_{\natural}^{k - |\a|/2} \, \a!^{s-1}  |a|_{s,R}^{k} R^{|\a|} \, \sum_{\a_1 + \cdots + \a_k = \a \atop \a_j \geq 1} \prod_{j=1}^{k} \binom{\a}{\a_1 , \ldots , \a_k}^{1-s} .
	\end{eqnarray*}

	\noindent Denote
	$$
		C_{s}(\a,k) = \sum_{\a_1 + \cdots + \a_k = \a \atop \a_j \geq 1} \prod_{j=1}^{k} \binom{\a}{\a_1 , \ldots , \a_k}^{1-s} .
	$$

	\noindent As $\binom{\a}{\a_1 , \ldots , \a_k} \geq 1$ and $s \geq 1$, there holds
	$$
		C_{s}(\a,k)
		\leq
		N(\a,k) .
	$$

	We put altogether all the inequalities:
	\begin{align*}
		& \left| \frac{1}{\a!\beta!} \,\d_{x}^{\a}\d_{\xi}^{\beta} \left( a_{\natural}^{-1/2} \right) \right| \\
		& \leq
		\sum_{1 \leq k \leq |\a| \atop 1 \leq k' \leq |\beta|} \frac{ |c_{k}|\,| c_{k'}| }{k! k'!} \,a_{\natural}^{-1/2-k} \, a_{\natural}^{k - |\a|/2} \, \a!^{s-1}  |a|_{s,R}^{k} R^{|\a|} \, C_s(\a,k) \, N(\beta,k') \,\langle \cdot \rangle^{ - |\beta|} \\
		& \leq
		a_{\natural}^{-1/2-|\a|/2} \a!^{s-1} R^{|\a|} \,\langle \cdot \rangle^{ - |\beta|} \, \sum_{1 \leq k \leq |\a| \atop 1 \leq k' \leq |\beta|} \frac{ |c_{k}|\,| c_{k'}| }{k! k'!} \,N(\beta,k') |a|_{s,R}^{k}  \, C_s(\a,k) \\
		& \leq
		b(t,x,\xi) \, a_{\natural}^{-|\a|/2} \a!^{s-1} R^{|\a|} \,\langle \cdot \rangle^{ - |\beta|} \, \left( \sum_{1 \leq k \leq |\a| } \frac{ |c_{k}|  }{k! } |a|_{s,R}^{k}  \, N(\a,k) \right) \, \left( \sum_{1 \leq k' \leq |\beta|} \frac{ |c_{k'}| }{k'!} \,N(\beta,k') \right) .
	\end{align*}

	\noindent By definition of the $c_{k}$ there holds
	\begin{eqnarray*}
		c_{k}
		& = &
		\prod_{j=0}^{k-1} \left( -1/2 - j \right) \\
		& = &
		(-1/2)^{k} \prod_{j=0}^{k-1} \left( 2j+1 \right) \\
		& = &
		\left( \frac{-1}{2} \right)^{k} \frac{ (2k)! }{ \prod_{j=0}^{k-1} (2(j+1)) } \\
		& = &
		\left( \frac{-1}{4} \right)^{k} \frac{ (2k)! }{ k! } .
	\end{eqnarray*}

	\noindent In particular there holds $|c_k| \leq k! $ by Stirling's inequality.
	This implies that
	$$
		\sum_{1 \leq k \leq |\a| } \frac{ |c_{k}|  }{k! } |a|_{s,R}^{k}  \, N(\a,k) \leq (1+|a|_{s,R})^{|\a|}
	$$

	\noindent and
	$$
		\sum_{1 \leq k' \leq |\beta|} \frac{ |c_{k'}| }{k'!} \,N(\beta,k') \leq 2^{|\beta|} .
	$$

	\noindent Finally there holds
	\begin{align*}
		& \left| \frac{1}{\a!\beta!} \,\d_{x}^{\a}\d_{\xi}^{\beta} \left( a_{\natural}^{-1/2} \right) \right| \\
		& \leq
		b(t,x,\xi) \, b^{|\a|} \a!^{s-1} \left( R(1 + |a|_{s,R} ) \right)^{|\a|} 2^{|\beta|}\,\langle \cdot \rangle^{ - |\beta|}
	\end{align*}

	\noindent which suffices to end the proof.

\end{proof}

\medskip

We do not use Lemma \ref{4.lemma.glaeser.local} here, but include it since it may prove useful in further work on weakly hyperbolic systems.
We note that a local statement can be deduced from Lemma \ref{4.lemma.glaeser.global}, using a $C^{\infty}$ nonnegative function $\varphi$ with compact support $\overline{B}_{r}(x_0)$, and equals to $1$ in $\overline{B}_{r'}(x_0)$ for some $r' <r$.
We may then extend any locally defined, nonnegative function into a globally defined, nonnegative one.

We first introduce some notations.
For any domain $\mathcal{D} \subset B_r(x_0) $ and $j \in \N$, we denote
$$
	M_j(f ; \mathcal{D}) = \sup \left\{ |\d_x^{\a} f| \,:\, x \in \mathcal{D} \, , \,  |\a| = j \right\} .
$$

\noindent For any $0<r'<r$ we define
$$
	\mathcal{C}_{r',r}(x_0) = \left\{ x \in \R^{d} \, : \, r' < |x-x_0| < r\right\} .
$$
\begin{lemma}[Local Glaeser inequality]
	\label{4.lemma.glaeser.local}
	Let $f:B_{r}(x_0) \to \R$ be a nonnegative $C^2$ function.
	Then
	\be
		\label{4.ineq.glaeser.local}
		\forall x\in \overline{B}_{r'}(x_0) \;,\quad |\d_{x}f(x)|^2 \leq G(f;x_0,r',r) f(x)
	\ee

	\noindent for any $r'<r$.
	The local Glaeser's constant $G(f;x_0,r',r)$ is defined by
	\be
		\label{4.constant.glaeser.local}
		G(f;x_0,r',r) = 2M_{2}(f;B_{r}(x_0)) + \frac{4}{r-r'} M_{1} \left( f;\mathcal{C}_{r',r}(x_0) \right) + \frac{4}{(r-r')^2} M_{0} \left( f; \mathcal{C}_{r',r}(x_0) \right) .
	\ee
\end{lemma}

\begin{proof}
	Let $\varphi$ be a $C^{\infty}$ function with compact support $\overline{B}_{r}(x_0)$, satisfying also $ 0 \leq \varphi \leq 1$ and $\varphi(x) = 1$ for all $x\in\overline{B}_{r'}(x_0)$.
	Then the function $f\varphi$ satisfies the conditions for applying Lemma \ref{4.lemma.glaeser.global}.
	Hence \eqref{4.ineq.glaeser.global} leads to
	$$
		|(f\varphi)'(x)|^2 \leq 2 M_{2}(f\varphi;\R^{n}) f(x)\varphi(x)
	$$

	\noindent for all $x$ in $\R^{n}$.
	As $\varphi$ is identically one in $\overline{B}_{r'}(x_0)$, there holds
	$$
		|f'(x)|^2 \leq 2 M_{2}(f\varphi;\R^{n}) f(x)
	$$

	\noindent for all $x\in\overline{B}_{r'}(x_0)$.

	To end the proof we have to give an upper bound of $M_{2}(f\varphi;\R^{n})$, with respect to the distance $r-r'$.
	First there holds
	$$
		M_{2}(f\varphi;\R^{n}) \leq M_{2}(f;B_{r}(x_0)) + 2M_{1} \left( f; \mathcal{C}_{r',r}(x_0) \right) M_{1} (\varphi;\R^{n}) + M_{0} \left( f; \mathcal{C}_{r',r}(x_0) \right) M_{2}(\varphi;\R^{n}).
	$$

	\noindent Second, for any $x_{r'}$ such that $|x_{r'}| = r'$ we denote
	$$
		x_{r} = x_{0} + \frac{r}{r'}(x_{r'} - x_0)
	$$

	\noindent the only point of $\overline{B}_{r}(x_0)$ such that $|x_{r}-x_0| = r$ and $x_{r'}$ is in the interval $[x_0,x_{r}]$.
	By the mean value theorem there is $s\in[0,1]$ such that
	$$
		\varphi(x_{r}) - \varphi(x_{r'}) = (x_{r} - x_{r'}) \cdot \d_x \varphi\left( x_{0} + s \frac{r}{r'}(x_{r'} - x_0) \right)
	$$

	\noindent thus
	$$
		\left|\d_x \varphi\left( x_{0} + s \frac{r}{r'}(x_{r'} - x_0) \right)\right| = \frac{1}{r-r'}
	$$

	\noindent as $\varphi(x_{r}) = 0$, $\varphi(x_{r'}) = 1$ and $|x_{r} - x_{r'}| = r-r' $ and then
	$$
		M_{1}(\varphi) \geq \frac{1}{r-r'}.
	$$

	\noindent By the same way we can prove also that
	$$
		M_{2}(\varphi) \geq \frac{2}{(r-r')^2}.
	$$

	To end the proof, it suffices to construct $\varphi$ such that  the previous lower bound are equalities.
\end{proof}

\begin{remark}
	In the estimate \eqref{4.constant.glaeser.local} appears the distance $r-r'$.
	In the worst case, it is the distance between the neighborhood of $x_0$ such that the Glaeser inequality holds, and the possible point $\tilde{x}$ such that $f(\tilde{x}) = 0$ and $\d_{x} f(\tilde{x}) \neq 0$, at which Glaser inequality fails.

	For example, let take $f(x) = x$ in $[0,+\infty[$.
	Then $f'(x) = 1$ and there holds, for any $x_0>0$:
	$$
		\forall \,x\in[x_0,x_0+1] \;,\quad (f'(x))^2 \leq C(x_0) f(x) 
	$$

	\noindent with $C(x_0) = 1/x_0$.
	By comparison, the constant $G(f;r',r)$ of the previous Lemma verifies
	\begin{eqnarray*}
		G(f;r',r)
		& \leq &
		M_{2}(f;[x_0,x_0+1]) + \frac{2}{x_0} M_1(f;[0,x_0]) + \frac{2}{x_0^2} M_{0}(f;[0,x_0]) \\
		& \lesssim &
		\frac{1}{x_0}
	\end{eqnarray*}

	\noindent as $M_{0}(f;[0,x_0]) \leq x_0$.
\end{remark}


\subsection{Metrics in the phase space and pseudo-differential calculus}
\label{4.subsection.metrics}


We refer to the Chapter 2 of \cite{lerner2011metrics} for the basic definitions and expected properties of metrics in the phase space and associated symbols.
As we wish this paper to be self-contained, we give the few needed definitions, and prove that our metric $g$ and symbol $b$ satisfy them.
Hence we may use all properties of pseudo-differential calculus with symbols in $S(b,g)$ and other related classes.

\begin{lemma}[Admissibility of the metric]
	\label{lemma.admissibility.proof}
	The metric $g$ defined by \eqref{4.def.metric} is admissible, that is:
	\begin{enumerate}

		\item The metric $g$ is slowly varying (see Definition $2.2.1$ in {\rm \cite{lerner2011metrics}}), as there are $C>0$ and $r>0$ such that for all $X, Y, T \in \R\times\R$ there holds
		\be
			\label{slowly.ineq}
			g_X(X-Y) \leq r^2 \implies C^{-1} g_Y(T) \leq g_X(T) \leq C g_Y(T) .
		\ee

		\item The metric $g$ satisfies the uncertainty principle (see Section $2.2.3$ and specifically $2.2.12$ in {\rm \cite{lerner2011metrics}}), that is
		\be
			\label{uncertainty}
			\lambda(t,x,\xi) \geq 1 \quad , \quad \forall \, t,x,\xi .
		\ee
		\noindent where $\lambda$ is defined by \eqref{4.def.lambda}.

		\item The metric $g$ is temperate (see Lemma $2.2.14$ in {\rm \cite{lerner2011metrics}}), that is there are $C>0$ and $N >0$ such that
		\be
			\label{temperate.ineq}
			\frac{g_X(T)}{g_Y(T)} \leq C \left(	1 + g_X^{\sigma} (X-Y)\right)^{N} \quad , \quad \forall \, X,Y,T \text{ in } \R\times\R .
		\ee

		\noindent The metric $g_X^{\sigma}$ is defined by
		$$
			g_X^{\sigma}(Y) =  \langle X_2 \rangle^2|Y_1|^2 + a_{\natural}(t,X)|Y_2|^2 .
		$$
	\end{enumerate}
\end{lemma}

\begin{proof}
	We follow here partially the proof of Lemma 3.1 in \cite{colombini2007second}.
	We remind that $g_X(Y)$ reads
	$$
		g_X(Y) = \frac{|Y_1|^2}{a_{\natural}(t,X)} + \frac{|Y_2|^2}{\langle X_2 \rangle^2}
	$$

	\noindent with $X = (X_1,X_2)$ (the first component has to be seen as $x\in\R$ and the second component as $\xi\in\R$).

	$1.$ Assume that there holds
	\be
		\label{local.slowly.1}
		\frac{|X_2 - Y_2|^2}{\langle X_2 \rangle^2} \leq r_2^2
	\ee

	\noindent for some $r_2>0$.
	This implies in particular that
	\begin{eqnarray*}
		|Y_2|^2
		& \leq &
		2|X_2|^2 + 2|X_2 - Y_2|^2  \\
		& \leq &
		2(1 + r_2^2)|X_2|^2
	\end{eqnarray*}

	\noindent hence $\langle Y_2 \rangle \leq \sqrt{2}(1 + r_2^2)^{1/2}\langle X_2 \rangle$.
	The same way we prove
	$$
		|X_2|^2 \leq 2r_2^2|X_2|^2 + 2|Y_2|^2
	$$

	\noindent which leads to $(1 -2 r_2^2)^{1/2} \langle X_2 \rangle \leq \sqrt{2}\langle Y_2 \rangle$ as soon as $r_2 < 1/\sqrt{2}$.
	We have just proved the following
	$$
		\frac{|X_2 - Y_2|^2}{\langle X_2 \rangle^2} \leq r_2^2 \implies C_2^{-1} \frac{|T_2|^2}{\langle Y_2 \rangle^2} \leq \frac{|T_2|^2}{\langle X_2 \rangle^2} \leq C_2 \frac{|T_2|^2}{\langle Y_2 \rangle^2} \quad , \quad \forall \,T_2 \in \R
	$$

	\noindent with $r_2 <1/\sqrt{2}$ and $C_2 >0$ depending only on $r_2$.

	Next, we consider the part $\frac{|Y_1|^2}{a_{\natural}(t,X)}$ of the metric $g$.
	Assume that there holds
	\be
		\label{local.slowly.2}
		\frac{|X_1 - Y_1|^2}{a_{\natural}(t,X)} \leq r_1^2
	\ee

	\noindent for some $X$ and $Y$ in $\R\times\R$ and $r_1 >0$, and that \eqref{local.slowly.1} still holds with $r_2 <1/\sqrt{2}$.
	We aim to compare $a_{\natural}(t,X)$	and $a_{\natural}(t,Y)$.
	As $a(t,x)$ is smooth with respect to $x$, there holds
	$$
		a(t,X_1) = a(t,Y_1) + (X_1 - Y_1)\d_xa(t,Y_1) + R
	$$

	\noindent with $|R| \leq |X_1 - Y_1|^2 \sup \left| \d_x^2 a(t,x) \right| $.
	Note that, as $a$ has compact support in $x$, that $\sup \left| \d_x^2 a(t,x) \right| $ is independent of $X_1$ and $Y_1$.
	Then, using \eqref{local.slowly.2} we get
	$$
		|R| \leq C_2' a_{\natural}(t,X) r_1^2
	$$

	\noindent and
	$$
		\left| (X_1 - Y_1)\d_xa(t,Y_1) \right| \leq C_2'' a_{\natural}(t,X)^{1/2} r_1 a_{\natural}(t,Y)^{1/2}
	$$

	\noindent thanks to Lemma \ref{4.lemma.glaeser.und.a}.
	Note that neither $C_2'$ nor $C_2''$ depend on $X$ or $Y$.
	Then, by positivity of $a$, there holds
	\begin{eqnarray*}
		a(t,X_1)
		& \leq &
		a_{\natural}(t,Y) + C_2'' a_{\natural}(t,X)^{1/2} r_1 a_{\natural}(t,Y)^{1/2} + C_2' a_{\natural}(t,X) r_1^2 \\
		& \leq &
		2 a_{\natural}(t,Y) + \left(C_2' + C_2''^2/4 \right) r_1^2 a_{\natural}(t,X) .
	\end{eqnarray*}

	\noindent As soon as $r_1$ satisfies $C_2' + C_2''^2/4 < r_1^{-2}$, we get
	$$
		a_{\natural}(t,X) \leq C_2 a_{\natural}(t,Y)
	$$

	\noindent which suffices to prove finally \eqref{slowly.ineq}.

	\medskip

	2. The uncertainty principle for general metrics in the phase space reads in our case \eqref{uncertainty}.
	As the inequality should hold for all times, in particular for $t=0$ this leads to
	$$
		\lambda(0,x,\xi) = \langle \xi \rangle^{-c/2} \langle \xi \rangle \geq 1 \quad , \quad \forall \, \xi \in \R
	$$

	\noindent hence $c \leq 2$.
	As $\lambda$ is decreasing as time goes by, this inequality is sufficient to ensure \eqref{uncertainty} for all times.

	\medskip

	3. As $g$ is slowly varying, inequality \eqref{temperate.ineq} is satisfied if $g_X(X-Y) \leq r^2$.
	Assume then that $g_X(X-Y) > r^2$.
	As we note that $g^{\sigma}_X(X-Y) = a_{\natural}(t,X) \langle X_2 \rangle^2 g_X(X-Y) $, this implies that
	\begin{eqnarray}
		g^{\sigma}_X(X-Y)
		& \geq &
		a_{\natural}(t,X) \langle X_2 \rangle^2 r^2 \nonumber \\
		& \geq &
		\langle X_2 \rangle^{2-c} r^2 \label{local.gsigma.1}
	\end{eqnarray}

	\noindent by nonnegativity of $a$.

	For all $T\in\R\times\R$, there holds
	$$
		\frac{g_X(T)}{g_Y(T)} \leq \max \left( \frac{a_{\natural}(t,Y)}{a_{\natural}(t,X)} \, , \, \frac{\langle Y_2 \rangle^2}{\langle X_2 \rangle^2} \right) .
	$$

	\noindent To prove \eqref{temperate.ineq}, it suffices then to prove that
	\be
		\label{local.temp.ineq.1}
		\frac{a_{\natural}(t,Y)}{a_{\natural}(t,X)} \leq C' \left( 1 + g_X^{\sigma} (X-Y)\right)^{N'}
	\ee

	\noindent for some $C' >0$, $N'>0$ and
	\be
		\label{local.temp.ineq.2}
		\frac{\langle Y_2 \rangle^2}{\langle X_2 \rangle^2} \leq C'' \left( 1 + g_X^{\sigma} (X-Y)\right)^{N''}
	\ee

	\noindent for some $C'' >0$, $N''>0$.

	Consider first the former inequality.
	We proceed as for the point $1.$, proving there is a constant $B>0$ such that
	$$
		a(t,Y_1) = a(t,X_1) + B\left( a(t,X_1) + |X_1 - Y_1|^2 \right).
	$$

	\noindent Here, the control of $|X_1 - Y_1|^2$ is done by $g_X^{\sigma} (X-Y)$, as there holds
	\begin{eqnarray*}
		|X_1 - Y_1|^2
		& \leq &
		g_X^{\sigma} (X-Y) \\
		& \leq &
		C \langle X_2 \rangle^{-c} g_X^{\sigma} (X-Y)^{1+ c/(2-c)}
	\end{eqnarray*}

	\noindent thanks to \eqref{local.gsigma.1}.
	Then there is $B'>0$ such that
	$$
		a(t,Y_1) \leq B' a_{\natural}(t,X) \left( 1 + g_X^{\sigma} (X-Y) \right)^{1+ c/(2-c)}.
	$$

	\noindent As there holds $\langle Y_2 \rangle^{-c} \leq \langle Y_2 \rangle^{-c}(1 + |X_2 - Y_2|)^{c}$, we get
	\begin{eqnarray*}
		\langle Y_2 \rangle^{-c}
		&\leq &
		\langle Y_2 \rangle^{-c}(1 + |X_2 - Y_2|)^{c} \\
		&\leq &
		\langle Y_2 \rangle^{-c}\left( 1 + g_X^{\sigma} (X-Y)^{1/2+ c/2(2-c)} \right)^{c}
	\end{eqnarray*}

	\noindent hence
	$$
		a_{\natural}(t,Y) \leq C' a_{\natural}(t,X) \left( 1 + g_X^{\sigma} (X-Y) \right)^{N'} .
	$$

	\noindent for some $C'>0$ and $N'>0$.

	We may prove the second inequality \eqref{local.temp.ineq.2} by the same way, and end the proof.
\end{proof}

\begin{lemma}[Admissibility of the symbol $b$]
	The symbol $b$ defined by \eqref{4.def.b} is an admissible weight for the metric $g$, that is there are $C>0$ and $N >0$ such that
		\be
			\label{admissible.weight.ineq}
			\frac{b(t,X)}{b(t,Y)} \leq C \left(	1 + g_X^{\sigma} (X-Y)\right)^{N} \quad , \quad \forall \, X,Y \text{ in } \R\times\R .
		\ee
\end{lemma}

\begin{proof}
	In the course of the previous Lemma, we prove inequality \eqref{local.temp.ineq.1}.
	In view of the definition of an admissible weight, this means exactly that $a_{\natural}$ is an admissible weight for $g$.
	Lemma $2.2.22$ with $f(t) = t^{-1/2}$ and the fact that $a_{\natural} \in S(a_{\natural},g)$ implies then that $b = a_{\natural}^{-1/2}$ is also an admissible weight.

\end{proof}

For an admissible weight $M$ on $\R^{d}\times\R^{d}$, we introduce the classes of symbols $S(M,g)$ associated to the metric $g$:
\begin{defi}[Definition of classes of symbols]
	\label{4.defi.class.symbols}
	The space of symbols $S(M,g)$ is defined as the set of $C^{\infty}$ functions $f(t,\cdot)$ on $\R^{d}\times\R^{d}$ such that, for all $(\a,\beta)$ in $\N^d\times\N^d$, there is $C_{\a,\beta} >0$ such that
	$$
		\left| \d_{x}^{\a}\d_{\xi}^{\beta} f(t,x,\xi) \right| \leq C_{\a,\beta} M(x,\xi) \, b(t,x,\xi)^{|\a|} \langle \xi \rangle^{-|\beta|}
	$$

	\noindent uniformly in $(t,x,\xi)$.
\end{defi}


%
%
%
%
%
%



	\bibliographystyle{alpha}
	\bibliography{onset_insta}

\end{document}